\documentclass[11pt]{article}

\usepackage[margin=.9in]{geometry}
\usepackage[utf8]{inputenc}
\usepackage[T1]{fontenc}
\usepackage{amsmath,amssymb,amsthm,mathtools, todonotes, booktabs}
\usepackage{bbm}
\usepackage{enumitem}
\usepackage[numbers,sort&compress]{natbib}
\usepackage[colorlinks=true,linkcolor=blue,citecolor=blue,urlcolor=blue]{hyperref}
\usepackage[nameinlink,capitalise,noabbrev]{cleveref}
\usepackage{mathrsfs, enumitem}
\usepackage{algorithm}
\usepackage{algpseudocode}

\numberwithin{equation}{section}

\newtheorem{theorem}{Theorem}[section]
\newtheorem{corollary}[theorem]{Corollary}
\newtheorem{lemma}[theorem]{Lemma}

\newtheorem{assumption}[theorem]{Assumption}
\newtheorem{definition}[theorem]{Definition}
\theoremstyle{remark}
\newtheorem{remark}[theorem]{Remark}

\newcommand{\R}{\mathbb R}
\newcommand{\N}{\mathbb N}
\newcommand{\E}{\mathbb E}

\newcommand{\calE}{\mathcal E}

\newcommand{\supp}{\operatorname{supp}}
\newcommand{\dd}{\,\mathrm d}
\newcommand{\eps}{\varepsilon}
\newcommand{\norm}[1]{\left\lVert #1\right\rVert}
\newcommand{\abs}[1]{\left\lvert #1\right\rvert}

\newcommand{\dOm}{\partial\Omega}

\title{Operator-Split Bayesian Learning for Elliptic PDEs with Unequal Interior and Boundary Data}
\author{\Large Emmanuel E. Oguadimma\thanks{oguadime@oregonstate.edu}\\
Department of Mathematics, Oregon State University
}
\date{}

\begin{document}
\maketitle

\begin{abstract}
We propose an operator-split Bayesian learning framework for second-order uniformly elliptic Dirichlet problems with unequal numbers of interior and boundary observations. The data consist of noisy measurements of the source in the domain and noisy measurements of the boundary values. Independent Bayesian neural-network (BNN) priors are assigned to these two quantities, and the resulting product posterior is pushed forward through the elliptic solution operator. We prove that the posterior induced by this construction contracts around the true solution. The contraction radius separates a domain contribution, governed by the second-order elliptic operator, from a boundary contribution, governed by the intrinsic dimension of the boundary. Together with the minimax lower bound of \cite{ZhaoLu2026}, this yields a near-minimax upper bound up to logarithmic factors. Our numerical experiments illustrate the propagation of source and boundary uncertainty and the effects of unequal sampling budgets on the posterior reconstruction.
\end{abstract}
\noindent\textbf{Keywords:}
Bayesian neural networks; elliptic PDEs; posterior contraction; operator splitting; uncertainty quantification.

\section{Introduction}
Many problems in science and engineering are formulated by differential equations that describe how physical, biological, and engineered systems evolve \citep{Evans2010,QuarteroniSaccoSaleri2007,oguadimma2026foundational}. In special cases, these equations can be solved analytically, and their solutions often reflect qualitative structures such as stability, conservation laws, invariants, and symmetries \citep{oguadimma2026foundational}. For most models of practical interest, however, closed-form solutions do not exist, so one relies on numerical approximations. Numerical methods such as finite differences, finite elements, and spectral methods have therefore formed the foundation of scientific computing. More recently, neural network approximations have emerged as mesh-free and highly expressive alternatives, motivated in part by high-dimensional problems, inverse problems, and the difficulty of repeatedly building problem-specific discretizations \citep{SirignanoSpiliopoulos2018,Yu2018,RaissiPerdikarisKarniadakis2019,obieke2025structure,obieke2026structure,oguadimma2026operator}. Physics-informed neural networks (PINNs) approximate solutions of differential equations by fitting neural networks against losses that combine equation residuals, boundary or initial conditions, and available data \citep{RaissiPerdikarisKarniadakis2019}.

The traditional PINNs \cite{RaissiPerdikarisKarniadakis2019} typically produce a point estimator, and therefore do not directly quantify uncertainty in the learned solution. This limitation becomes important when measurements are noisy, sparse, or drawn from physical experiments with uncertain inputs or parameters. Bayesian inference offers a principled framework for uncertainty quantification by treating unknown functions \cite{RasmussenWilliams2006}, model parameters \cite{Stuart2010}, and neural-network weights \cite{Neal1996,MacKay1992} as random quantities endowed with prior distributions. Bayesian PINNs have recently been proposed for uncertainty quantification in forward and inverse problems governed by differential equations \citep{YangMengKarniadakis2021}. Rigorous posterior contraction results for elliptic problems are more recent and have primarily been developed for constructions that place a prior directly on the solution and use balanced interior and boundary sample sizes \cite{ZhaoLu2026}.

\subsection{Summary of Contributions}

In this paper, we construct an operator-split Bayesian learning framework for elliptic Dirichlet problems with noisy and imbalanced interior and boundary measurements. Unlike direct Bayesian PINN formulations that place a prior on the solution itself \cite{YangMengKarniadakis2021, ZhaoLu2026, SunMukherjeeAtchade2024}, the proposed construction assigns separate Bayesian neural-network priors to the interior source and boundary data and then propagates the resulting uncertainty through the elliptic solution operator.

The main contributions are as follows:
\begin{enumerate}[leftmargin=*, label=\arabic*.]
    \item We formulate independent Bayesian regression models for the interior source and the boundary observations, allowing the two data sets to have different sample sizes and to be analyzed at their respective dimensions.

    \item We combine the source and boundary posteriors through a product posterior and push this distribution forward through the elliptic solution map, thereby inducing a posterior distribution on the PDE solution.

    \item We prove posterior contraction in the physics-informed loss for sufficiently H\"older-smooth solutions. The resulting contraction radius separates the contribution from the \(d\)-dimensional interior observations from that of the \((d-1)\)-dimensional boundary observations.

    \item We show that the resulting two-sample contraction upper bound agrees, up to logarithmic factors, with the minimax lower bound of \cite{ZhaoLu2026}.

    \item We derive a boundary sampling condition under which the boundary contribution does not dominate the interior contribution, providing guidance for allocating unequal interior and boundary sampling budgets.

    \item We present numerical experiments that illustrate the propagation of source and boundary uncertainty and the effect of varying the two sample sizes on the posterior reconstruction.
\end{enumerate}



\subsection{Paper Organization}
The remainder of this paper is organized as follows. 
\Cref{sec:related} reviews related work on neural-network approximations of PDEs, Bayesian PINNs, Bayesian neural-network posterior contraction, and boundary penalties in residual minimization. 
\Cref{sec:preliminaries} introduces the notation, elliptic problem, neural-network classes, physics-informed loss, Bayesian formulation, posterior contraction framework, and theoretical results used in the analysis. 
\Cref{sec:split} constructs the operator-split Bayesian posterior, including the elliptic solution operators, boundary parametrization, product posterior, and induced pushforward posterior on the solution space. 
\Cref{sec:main-results} establishes the main posterior contraction theorem, its near-minimax comparison, and the boundary sampling condition. 
\Cref{sec:numerical-experiments} presents numerical experiments for one- and two-dimensional elliptic problems. 
Finally, \Cref{sec:conclusion} summarizes the main findings and discusses possible extensions.

\subsection*{Acknowledgment.}
We would like to thank Yuxuan Zhao for helpful discussions.

\subsection{Related work}
\label{sec:related}

\paragraph{Convergence theory for residual and variational neural approximations.}
Neural-network approximations of differential equations are often analyzed through residual or variational losses. PINNs enforce the governing equation by penalizing sampled residuals and boundary discrepancies \citep{RaissiPerdikarisKarniadakis2019}. The deep Galerkin method uses a residual formulation, whereas the Deep Ritz method is based on the corresponding energy principle \citep{SirignanoSpiliopoulos2018,Yu2018}. Error estimates for these methods typically combine neural-network approximation bounds, statistical sampling estimates, and stability estimates for the underlying differential operator \citep{MishraMolinaro2023,DeRyckMishra2022,ShinZhangKarniadakis2023}. For elliptic problems, \cite{LuChenLuYingBlanchet2022} established minimax-optimal rates for PINNs and the Deep Ritz method under homogeneous Dirichlet boundary conditions. These results concern frequentist point estimators obtained from empirical residual or variational minimization. Our work instead studies posterior contraction for a Bayesian physics-informed construction with noisy interior source observations and noisy nonhomogeneous boundary observations. The unequal sample sizes in the two data sources lead to separate \(d\)-dimensional interior and \((d-1)\)-dimensional boundary rates.

\paragraph{Bayesian PINNs.}
Bayesian PINNs were introduced by \cite{YangMengKarniadakis2021} as a probabilistic extension of PINNs for forward and inverse PDE problems with noisy data. Instead of returning only one trained neural-network solution, the Bayesian formulation places a posterior distribution on solutions and physical parameters, making it possible to report uncertainty in the learned fields. This framework has been used for nonlinear dynamical systems \citep{LinkaSchaferMengZouKarniadakisKuhl2022}, subsurface tomography \citep{GouZhangZhuGao2023}, flow-field tomography \citep{MolnarGrauer2022}, and turbulent-flow uncertainty quantification \citep{ShuklaZouKaeuferTriantafyllouKarniadakis2026}. Rigorous contraction theory is more recent. \cite{SunMukherjeeAtchade2024} study Bayesian PINNs for inverse problems using noisy observations of the solution itself, whereas \cite{ZhaoLu2026} prove posterior contraction for elliptic equations with noisy measurements of the interior source and nonhomogeneous Dirichlet boundary data. In this paper, we build on the latter observation model, but modify the posterior construction. 

\paragraph{Bayesian neural-network posterior contraction.}
Bayesian neural networks place priors on neural-network weights and use the resulting posterior distribution to quantify uncertainty in predictions \citep{MacKay1995,Neal1995}. Their asymptotic theory is commonly studied through posterior contraction, where one proves that the posterior concentrates around the true regression function at a rate determined by the prior mass, testing, and entropy or sieve estimates \citep{GhosalVanderVaart2007,GhosalVanderVaart2017}. For deep Bayesian neural networks, \cite{PolsonRockova2018} proved posterior concentration for sparse ReLU networks with spike-and-slab priors, building on approximation theory for deep neural networks. Later work established contraction guarantees for masked Bayesian neural networks \citep{KongYangLeeOhnBaekKim2023}, Besov function classes \citep{LeeLee2022,EgelsCastillo2025,LeeLinParkJeong2025}, and fully connected networks with non-sparse general weight priors \citep{KongKim2025}. 

\paragraph{Boundary penalties in residual minimization.}
For a boundary-value problem on a spatial domain \(\Omega\) with boundary \(\partial\Omega\), the boundary terms in physics-informed losses require separate control from the interior residual terms. As shown in \cite{MullerZeinhofer2022}, imposing boundary conditions through an \(L^2(\partial\Omega)\) penalty leads to a loss of $3/2$ derivatives in residual-minimization estimates. In particular, approximation at the \(H^2(\Omega)\) level gives error control only in \(H^{1/2}(\Omega)\), whereas enforcing the boundary condition exactly avoids this loss and yields a sharper residual error estimate. The posterior contraction result of \cite{ZhaoLu2026} is formulated in a physics-informed loss that includes the same \(L^2(\partial\Omega)\) boundary penalty. We retain that loss, but separate the statistical boundary component from the interior source component before applying the elliptic solution operator. This separation is what allows the boundary contribution to contract at its intrinsic \((d-1)\)-dimensional rate.

\section{Preliminaries}
\label{sec:preliminaries}

\subsection{Notation}

We write \(\R\), \(\N\), and \(\N_0\) for real numbers, positive integers, and nonnegative integers, respectively. Let \(\Omega\subset\R^m\) be a bounded domain with boundary \(\partial\Omega\). The volume measure on \(\Omega\) is denoted by \(\dd x\), and the surface measure on \(\partial\Omega\) by \(\dd S\). For $\mathbf{z} = (z_1, z_2, \cdots, z_m)^\top \in \Omega$, we denote 
\[
\|\mathbf z\|_{\ell^p}
:=
\left(\sum_{j=1}^m |z_j|^p\right)^{1/p}, \quad \|\mathbf z\|_{\ell^0} := \sum_{j=1}^m \mathbb I(z_j \neq 0), \quad \text{and} \quad \|\mathbf z\|_{\ell^\infty}=\max_{1\le j\le m}|z_j|,
\]
for $1\le p<\infty$. For $\mathbf z^{(1)}, \mathbf z^{(2)} \in \Omega$, we define the max$(\cdot, \cdot)$ operator by 
\[
    \text{max}(\mathbf z^{(1)}, \mathbf z^{(2)}) = \left(\text{max}\left(z_1^{(1)}, z_1^{(2)}\right), \text{max}\left(z_2^{(1)}, z_2^{(2)}\right), \cdots \text{max}\left(z_m^{(1)}, z_m^{(2)}\right)\right).
\]
Given real-valued functions \(v:\Omega\to\R\) and \(w:\partial\Omega\to\R\), we denote 
\[
   \norm{v}_{L^p(\Omega)}=\left(\int_\Omega |v(\mathbf z)|^p\,\dd x\right)^{1/p},\quad \norm{w}_{L^p(\partial\Omega)}=\left(\int_{\partial\Omega}|w(\mathbf y)|^p\,\dd S(\mathbf y)\right)^{1/p},
\]
for $1\le p<\infty.$ 

Let P\(_{\mathbf Z}\) be a probability measure on a compact set
\(\mathcal Z\subset \mathbb R^m\). For any measurable function
\(h:\mathcal Z\to \mathbb R\) and \(1\le p<\infty\), we write
\begin{align}\label{pdf}
\norm{h}_{L^p(\text{P}_{\mathbf Z})}
:=
\left(\int_{\mathbf Z \in \mathcal Z} |h(\mathbf z)|^p\,\dd \mathrm{P}_{\mathbf Z}\right)^{1/p} \quad \text{and} \quad \|h\|_{L^\infty} := \sup_{\mathbf z \in \mathcal{Z}} |h(\mathbf{z})|
\end{align}
We write \(a_n \lesssim b_n\) when \(a_n \le C b_n\) for some universal constant independent of \(n \in \mathbb{N}\), and \(a_n \asymp b_n\) when both \(a_n \lesssim b_n\) and \(b_n \lesssim a_n\) hold for sequences \(\{a_n\}\) and \(\{b_n\}\). 

For a multi-index \(\boldsymbol{\alpha} \in \mathbb{N}_0^m\), we denote $\partial^{\boldsymbol{\alpha}} = \partial_1^{\alpha_1} \cdots \partial_m^{\alpha_m}.$ We assume the true function belongs to the \(\beta\)-H\"older class \(\mathcal{H}_m^\beta(K)\), defined by
\[
\mathcal{H}_m^\beta(K)
:=
\left\{
h : \mathcal{Z} \to \mathbb{R} \;:\; \|h\|_{\mathcal{H}^\beta} \le K
\right\},
\]
where the H\"older norm \(\|\cdot\|_{\mathcal{H}^\beta}\) is given by
\[
\|h\|_{\mathcal{H}^\beta}
:=
\sum_{\boldsymbol{\alpha}:\,\|\boldsymbol{\alpha}\|_1 < \beta}
\|\partial^{\boldsymbol{\alpha}} h\|_{L^\infty}
+
\sum_{\boldsymbol{\alpha}:\,\|\boldsymbol{\alpha}\|_1 = \lfloor \beta \rfloor}
\sup_{\mathbf{z}_1 \neq \mathbf{z}_2 \in \mathcal{Z}}
\frac{
\big|\partial^{\boldsymbol{\alpha}} h(\mathbf{z}_1) - \partial^{\boldsymbol{\alpha}} h(\mathbf{z}_2)\big|
}{
\|\mathbf{z}_1 - \mathbf{z}_2\|_\infty^{\beta - \lfloor \beta \rfloor}
}.
\]

\subsection{General Elliptic Equation}\label{elipsec}
Let $\Omega\subset \mathbb R^d$ be a bounded open subset with smooth boundary $\partial\Omega$. We consider the second-order elliptic equation with Dirichlet boundary condition
\begin{align}
\mathcal Lu &= f,\quad  \text{in } \Omega, \label{eq:elliptic-model}\\
u &= g,  \quad \text{on } \partial\Omega, \nonumber
\end{align}
where $\mathcal Lu := -\operatorname{div}(A\nabla u)+Vu.$ We assume that the matrix coefficient, $A:\Omega\to \mathbb R^{d\times d}$, and potential, $V$, are sufficiently smooth. That is, $A, V \in C^\infty(\overline{\Omega})$, and there exist constants \(r_{\min},C_A,V_{\min},V_{\max}>0\) such that $r_{\min}I\preceq A(x),\,\sup_{x\in\Omega}\norm{A(x)}_\infty\vee \norm{\nabla A(x)}_\infty\le C_A,\, \text{and }0<V_{\min}\le V(x)\le V_{\max}$ over $\Omega$. 

In addition, we assume that \(f \in C^{\beta-2}(\Omega)\) and \(g \in C^\beta(\partial\Omega)\) with \(\beta > 2\). By the standard elliptic regularity theory \cite{Grisvard2011}, equation \eqref{eq:elliptic-model} admits a unique strong solution \(u^\ast \in C^\beta(\Omega)\).

\subsection{Deep Neural Networks}
Let $\mathbf r=\bigl(r^{(0)},r^{(1)},\ldots,r^{(L)},r^{(L+1)}\bigr)^\top
\in\mathbb N^{L+2}$ be the layer-width vector of the network, where \(L\in\mathbb N\) is the number of hidden layers,
\(r^{(0)}=m\) is the input dimension, \(r^{(L+1)}=1\) is the output dimension. We employ fully connected deep neural networks (DNNs) throughout this paper.

The \(l\)-th hidden layer contains
\(r^{(l)}\) neurons. For \(l=1,\ldots,L+1\), let $\mathcal A_l: \mathbb{R}^{r^{(l-1)}} \to \mathbb R^{r^{(l)}}$ be an affine map defined as $\mathcal A_l(\mathbf z)=W_l \mathbf z+\mathbf b_l,$ with $W_l\in\mathbb R^{r^{(l)}\times r^{(l-1)}}$ and $\mathbf b_l \in \mathbb R^{r^{(l)}}$. Given an activation function \(\boldsymbol \sigma:\mathbb R\to\mathbb R\), applied componentwise, the network output is
\begin{align}\label{nn}
h_{\boldsymbol \theta,\boldsymbol \sigma}^{\mathrm{DNN}}
=
\mathcal A_{L+1}
\circ
\boldsymbol \sigma 
\circ
\mathcal A_L
\circ
\cdots
\circ
\boldsymbol \sigma 
\circ
\mathcal A_1.
\end{align}
In the preceding equation, the parameter vector, $\boldsymbol{\theta}$, is given by $\boldsymbol{\theta}=(\boldsymbol{\theta}_w^\top,\boldsymbol{\theta}_b^\top)^\top$, which collects all the parameters of the DNN model, where
\begin{align*}
    \boldsymbol{\theta}_w=\bigl(\operatorname{vec}(W_1)^\top,
\ldots,\operatorname{vec}(W_{L+1})^\top
\bigr)^\top \quad \text{and} \quad  \boldsymbol{\theta}_b
=(\mathbf b_1^\top,\ldots,\mathbf b_{L+1}^\top)^\top,
\end{align*}
are the weight matrices and bias vectors, respectively. The total dimension of $\boldsymbol{\theta}$ is
\begin{align}\label{num}
    N := N(L,\mathbf r) = \sum_{l=1}^{L+1} \bigl(r^{(l-1)}+1\bigr)r^{(l)}.
\end{align}
We consider DNN architectures with fixed width, which means that each hidden layer contains the same number of neurons. Given an activation function \(\boldsymbol \sigma\), the number of hidden layers \(L\in\mathbb N\), and the hidden-layer width \(r\in\mathbb N\), we define the corresponding class of DNN functions, parameterized by \(\theta\in\mathbb R^{N}\), as follows.

\begin{definition}
Let \(\boldsymbol \sigma \) be an activation function, let \(L\in\mathbb N\) be the depth, and let \(r\in\mathbb N\) be the width. We define
\(\mathcal M_{\boldsymbol \sigma}^{\mathrm{DNN}}(L,r)\) as the class of fully connected DNNs with architecture $\bigl(L,(m,r,\ldots,r,1)^\top\bigr),$ and activation function \(\boldsymbol \sigma\). That is,
\[
\mathcal M_{\boldsymbol \sigma}^{\mathrm{DNN}}(L,r)
:=
\left\{
h : h=h_{\boldsymbol \theta,\boldsymbol \sigma}^{\mathrm{DNN}}
\text{ is a DNN with architecture }
\bigl(L,(m,r,\ldots,r,1)^\top\bigr)
\right\}.
\]
\end{definition}

\begin{definition}[Truncated DNN class]
For \(B\ge 1\) and $t\in \mathbb R$, define
\[
T_B(t)=\max\{-B,\min(t,B)\}.
\]
The truncated DNN class associated with
\(\mathcal M_{\boldsymbol\sigma}^{\mathrm{DNN}}(L,r)\) is
\[
\mathcal M_{\boldsymbol\sigma}^{\mathrm{DNN}}(L,r,B)
:=
\left\{
T_B\circ h:
h\in \mathcal M_{\boldsymbol\sigma}^{\mathrm{DNN}}(L,r)
\right\}.
\]
\end{definition}

\subsection{Physics-Informed Neural Networks}
In this section, we briefly recall the physics-informed neural network formulation for the elliptic Dirichlet problem \eqref{eq:elliptic-model}. A physics-informed neural network provides a residual-minimization approach to approximating solutions of differential equations.  

For \(u\in H^2(\Omega)\), we define the (population) physics-informed loss by
\begin{align}\label{eloss}
\mathcal E(u)=
\|\mathcal L u-f\|_{L^2(\Omega)}^2
+
\lambda\|u-g\|_{L^2(\partial\Omega)}^2,
\end{align}
for $\lambda\in (0, \infty).$ Equivalently, if \(u^\ast\) denotes the exact solution of \eqref{eq:elliptic-model}, then
\begin{align}\label{eploss1}
\mathcal E(u)
:=
\|\mathcal L(u-u^\ast)\|_{L^2(\Omega)}^2
+
\lambda
\|u-u^\ast\|_{L^2(\partial\Omega)}^2.
\end{align}

\subsubsection*{Discrete Physics-Informed Loss}

Let $X_i\stackrel{\mathrm{i.i.d.}}{\sim}\mathcal{U}(\Omega)$. We observe $f_i=f(X_i)+\eps_i$, which is the right hand side of the PDE \eqref{eq:elliptic-model}, with $ \eps_i\stackrel{\mathrm{i.i.d.}}{\sim}\mathcal N(0,\sigma_{\Omega,0}^2),$ where the noise variables are independent of the sampling locations. Similarly, let
\(Y_j\stackrel{\mathrm{i.i.d.}}{\sim}\mathcal U(\partial\Omega)\), and set $g_j=g(Y_j)+\eta_j,$ with $\eta_j\stackrel{\mathrm{i.i.d.}}{\sim}\mathcal N(0,\sigma_{\partial,0}^2)$, where the noise variables are independent of the the sampled boundary points. Denote $\mathcal D_\Omega^{(n_\Omega)}=\{(X_i,f_i)\}_{i=1}^{n_\Omega},\, \mathcal D_\partial^{(n_\partial)}=\{(Y_j,g_j)\}_{j=1}^{n_\partial},$ and $\mathcal D^{(n_\Omega,n_\partial)}= \mathcal D_\Omega^{(n_\Omega)} \cup \mathcal D_\partial^{(n_\partial)}.$

The empirical analogue of \eqref{eloss} is defined by
\begin{align}
\widehat{\mathcal E}
\bigl(u;\mathcal D^{(n_\Omega,n_\partial)}\bigr)
&:=
\frac{|\Omega|}{n_\Omega}
\sum_{i=1}^{n_\Omega}
\left|
\mathcal L u(X_i)-f_i
\right|^2
+
\lambda
\frac{|\partial\Omega|}{n_\partial}
\sum_{j=1}^{n_\partial}
\left|
u(Y_j)-g_j
\right|^2.
\label{eq:empirical-pinn-loss}
\end{align}

\subsection{Bayesian Deep Learning}

Let \(h_{\boldsymbol \theta,\boldsymbol \sigma}^{\mathrm{DNN}}\) be a DNN with architecture \((L,\mathbf r)\), as defined in \eqref{nn}, and let \(\boldsymbol \theta=(\theta^{(1)},\ldots,\theta^{(N)})^\top\in\mathbb R^{N}\) denote the vector of all weights and biases. In Bayesian deep learning, the network parameters are treated as random variables. We assign a prior distribution \(\Pi\) on \(\mathbb R^N\), and update this prior using the observed training data.

Given the training data \(\mathcal D^{(n)}\) and the likelihood function \(\mathscr L(\boldsymbol \theta\mid \mathcal D^{(n)})\), the posterior distribution over the network parameters is
\[
\Pi_n(A\mid \mathcal D^{(n)}) 
=
\frac{
\int_A \mathscr L(\boldsymbol \theta\mid \mathcal D^{(n)})\,d\Pi(\boldsymbol \theta)
}{
\int_{\mathbb R^{N}} \mathscr L(\boldsymbol \theta\mid \mathcal D^{(n)})\,d\Pi(\boldsymbol \theta)
},
\qquad A\subset \mathbb R^N.
\]
For a new test example \(\mathbf z \in \mathcal{Z}\), the Bayesian predictive distribution is obtained by averaging over the posterior distribution of the network parameters:
\[
p(y\mid \mathbf z,\mathcal D^{(n)})
=
\int_{\mathbb R^N}
p(y\mid \mathbf z,\theta)\,\Pi_n(d\boldsymbol \theta\mid \mathcal D^{(n)}).
\]
Since closed form solutions of the preceding integral do not usually exist, it is commonly approximated by Monte Carlo method. If
\(\boldsymbol \theta_1,\ldots,\boldsymbol \theta_B\sim \Pi_n(\cdot\mid \mathcal D^{(n)})\), which are the samples drawn from the posterior distribution, then
\[
p(y\mid z,\mathcal D^{(n)})
\approx
\frac{1}{B}
\sum_{b=1}^{B}
p(y\mid z,\boldsymbol \theta_b).
\]
In practice, such samples may be generated using stochastic-gradient Markov chain Monte Carlo methods, including stochastic-gradient Langevin dynamics \cite{WellingTeh2011}, stochastic-gradient Hamiltonian Monte Carlo \cite{ChenFoxGuestrin2014}, and cyclical stochastic-gradient MCMC \cite{ZhangLiZhangChenWilson2020}. Another common approach is variational inference, including practical variational inference for neural networks \cite{Graves2011} and Bayes by Backprop \cite{BlundellCornebiseKavukcuogluWierstra2015}.

\subsection{Posterior Concentration For Deep Learning}
Let the training data \(\mathcal D^{(n_\Omega,n_\partial)}\), consisting of interior measurements $\{(X_i,f_i)\}_{i=1}^{n_\Omega}$ and boundary measurements $\{(Y_j,g_j)\}_{j=1}^{n_\partial},$ be given by
\[
\mathcal D^{(n_\Omega,n_\partial)}
=
\mathcal D^{(n_\Omega)}
\cup
\mathcal D^{(n_\partial)}
=
\{(X_i,f_i)\}_{i=1}^{n_\Omega}
\cup
\{(Y_j,g_j)\}_{j=1}^{n_\partial}.
\]
For a solution \(u\), let \(p_u^{(n_\Omega)}\) denote the likelihood density of the interior data \(\mathcal D^{(n_\Omega)}\), and let P\(_u^{(n_\Omega)}\) be the corresponding probability distribution. Similarly, let \(q_u^{(n_\partial)}\) denote the likelihood density of the boundary data \(\mathcal D^{(n_\partial)}\), and let Q\(_u^{(n_\partial)}\) be its corresponding distribution.

Let \(\Pi\) be a prior distribution on a solution class \(\mathscr U\). By Bayes' theorem, the posterior mass of any measurable set \(\mathscr E\subset \mathscr U\) is
\[
\Pi(\mathscr E\mid \mathcal D^{(n_\Omega,n_\partial)})
=
\frac{
\int_{\mathscr E}
p_u^{(n_\Omega)}
q_u^{(n_\partial)}
\,d\Pi(u)
}{
\int_{\mathscr U}
p_u^{(n_\Omega)}
q_u^{(n_\partial)}
\,d\Pi(u)
}.
\]
The posterior contraction property describes how fast this posterior distribution concentrates around the true solution \(u^\ast\) as the sample sizes increase. For \(\epsilon>0\) and \(M>0\), define the \((M\varepsilon)\)-neighborhood of \(u^\ast\) by
\[
A_{\varepsilon,M}
=
\left\{
u\in\mathscr U:
\mathcal E(u)\le M^2\varepsilon^2
\right\}.
\]
The complement \(A_{\varepsilon,M}^c\) consists of solutions whose physics-informed error $\mathcal E(u)$ is larger than the radius \(M\varepsilon\).

A sequence \(\varepsilon_{n_\Omega,n_\partial}\) is called a posterior concentration rate if, for every sequence
\(M_{n_\Omega,n_\partial}\to\infty\),
\[
\Pi\!\left(
A_{\varepsilon_{n_\Omega,n_\partial},M_{n_\Omega,n_\partial}}^c
\mid
\mathcal D^{(n_\Omega,n_\partial)}
\right)
\to 0
\]
in P$_{u^\ast}^{(n_\Omega)}$Q$_{u^\ast}^{(n_\partial)}$ \text{-probability} as \(n_\Omega\wedge n_\partial\to\infty\). Equivalently,
\[
\text{P}_{u^\ast}^{(n_\Omega)}
\text{Q}_{u^\ast}^{(n_\partial)}
\left[
\Pi\!\left(
A_{\varepsilon_{n_\Omega,n_\partial},M_{n_\Omega,n_\partial}}^c
\mid
\mathcal D^{(n_\Omega,n_\partial)}
\right)
\right]
\to 0.
\]
This means that, with probability tending to one under the true data-generating law, the posterior assigns asymptotically negligible mass to functions whose physics-informed loss exceeds the contraction radius.

\subsection{Known Theoretical Results}
\label{subsec:known-results}

We first review the theoretical results underlying our posterior contraction analysis. The posterior concentration theorem of \cite{KongKim2025} provides contraction rates for fully connected Bayesian neural networks in nonparametric regression under broad weight priors, including Gaussian priors, and therefore justifies treating the interior and boundary observations as Bayesian regression problems without imposing sparsity assumptions. We then recall the Bayesian PINN contraction theory and minimax lower bound of \cite{ZhaoLu2026}, which form the main benchmark for our posterior contraction analysis of the elliptic boundary value problem studied in this paper.

\subsubsection{Nonparametric Gaussian Regression} 
Let \(\mathcal D^{(n)}=\{(X_i,Y_i)\}_{i=1}^n\) be independent observations generated from $X_i\sim \text{P}_X,\,Y_i\mid X_i\sim \mathcal N(h_0(X_i),\sigma_0^2)$ where P\(_X\) is a probability measure on a compact set \(\mathcal X\subset [-a,a]^d\). Here \(h_0:\mathcal X\to\mathbb R\) is the true regression function and \(\sigma_0^2>0\) is the variance of the noise. We assume that \(\|h_0\|_\infty\le B\) for $\mathcal H^\alpha_d$ for some \(B\ge 1\), \(\alpha>0\), and \(K\ge 1\).

For the Bayesian inference, we utilize a fully connected DNN with
Leaky-ReLU activation
\[
    \boldsymbol \rho_\nu(t)=\max(t,\nu t),\quad \nu\in[0,1),
\]
and write the corresponding network as
\(h_{\boldsymbol\theta,\rho_\nu}^{\mathrm{DNN}}\). We consider the probabilistic model
\[
    Y_i\mid X_i,\boldsymbol\theta,\sigma^2
    \sim
    \mathcal N\!\left(
        T_B\circ h_{\boldsymbol\theta,\rho_\nu}^{\mathrm{DNN}}(X_i),
        \sigma^2
    \right),
    \qquad i=1,\ldots,n .
\]
The network architecture, $(L_n,\mathbf r_n)$, is chosen according to the smoothness \(\alpha\),
the input dimension \(d\), and the sample size \(n\). Specifically, we take
\begin{align}\label{arch}
    L_n &= \lceil C_1\log n\rceil,\quad r_n = \left\lceil C_2 n^{\frac{d}{2(2\alpha+d)}}\right\rceil,\quad \mathbf r_n = (d,r_n,\ldots,r_n,1)^\top,
\end{align}
where \(C_1,C_2>0\) are constants. Consequently, the likelihood for the Bayesian regression model is
\[
    \mathscr L(\boldsymbol\theta,\sigma^2\mid \mathcal D^{(n)})
    =
    (2\pi\sigma^2)^{-n/2}
    \exp\left\{
        -\frac{1}{2\sigma^2}
        \sum_{i=1}^n
        \left(
            Y_i
            -
            T_B\circ h_{\boldsymbol\theta,\rho_\nu}^{\mathrm{DNN}}(X_i)
        \right)^2
    \right\}.
\]
We then assign a prior \(\Pi_n\) on
\(\boldsymbol\theta\in\mathbb R^{N_n}\), where $N_n$ is defined as 
\begin{align}\label{num1}
    N_n := (d+ 1)r_n + (L_n-1)(r_n + 1)r_n + (r_n + 1),
\end{align}
and an independent prior \(\Xi\) on the variance parameter \(\sigma^2\in\mathbb R_+\).

The following assumption gives a lower-bound condition on the density of the parameter prior. It ensures that the prior assigns sufficient mass to bounded neighborhoods of neural-network parameters.

\begin{assumption}\label{ass:KK-parameter-prior}
Let \(\Pi_n\) be the prior distribution of the parameter 
\(\boldsymbol \theta\in\mathbb R^{N_n}\), and suppose that \(\Pi_n\) has a Lebesgue density \(\pi_n\). For every \(\kappa>0\), assume that there exists a constant \(\delta_\kappa>0\), independent of \(N\), such that
\[
\inf_{\boldsymbol \theta\in[-\kappa,\kappa]^{N}}\pi(\boldsymbol \theta)
\ge \delta_\kappa^{N_n}.
\]
\end{assumption}

The next assumption is imposed on the prior of the variance. It requires positive prior density near the true noise variance and a mild tail bound.

\begin{assumption}\label{ass:KK-variance-prior}
Let \(\Xi\) be the prior distribution of \(\sigma^2\in\mathbb R_+\). Assume that \(\Xi\) has a Lebesgue density that is continuous and positive on \((0,2\sigma_0^2)\). Also assume that, for sufficiently large \(K\),
\[
\Xi(\sigma^2>K)\lesssim K^{-1}.
\]
\end{assumption}

The following theorem, due to \cite{KongKim2025}, shows that BNNs with general priors satisfying Assumptions \ref{ass:KK-parameter-prior} and \ref{ass:KK-variance-prior} attain minimax-optimal posterior concentration rates around the true regression function, up to a logarithmic factor.

\begin{theorem}[{\cite[Theorem 2]{KongKim2025}}]\label{thm:KK-regression}
Assume that \(h_0\in \mathcal H_d^\alpha(K)\) and \(\|h_0\|_\infty\le B\), where
\(K\ge 1\), \(\alpha>0\), and \(B\ge 1\). Consider the DNN model \(h_{\boldsymbol \theta,\rho_\nu}^{\mathrm{DNN}}\) with
architecture \eqref{arch}. Suppose that the parameter prior $\Pi_n$ and variance prior $\Xi$ satisfy Assumptions \ref{ass:KK-parameter-prior} and \ref{ass:KK-variance-prior}. Then the posterior distribution of \(T_B\circ h_{\boldsymbol \theta,\boldsymbol \rho_\nu}^{\mathrm{DNN}}\) and \(\sigma^2\) concentrates around \(h_0\) and \(\sigma_0^2\) at the rate $\varepsilon_{n}=n^{-\alpha/(2\alpha+d)}\log^\gamma (n),\, \gamma>2.$ More precisely, for every \(M_n\to\infty\),
\begin{align}\label{post}
\Pi_n\left(
(\boldsymbol \theta,\sigma^2):
\left\|T_B\circ h_{\boldsymbol \theta,\boldsymbol \rho_\nu}^{\mathrm{DNN}}-h_0\right\|_{L^2(\mathrm{P}_X)}
+
|\sigma^2-\sigma_0^2|
>
M_n\varepsilon_{n}
\;\middle|\;
\mathcal D^{(n)}
\right)
\to 0
\end{align}
in \(\mathrm{P}_0^{(n)}\)-probability as \(n\to\infty\), where \(\mathrm{P}_0^{(n)}\) denotes the probability law of the training data $\mathcal D^{(n)}$.
\end{theorem} 

\subsubsection{Bayesian PINN Posterior Contraction}

We recall the Bayesian PINN contraction result of \cite{ZhaoLu2026}. We begin by defining the prior  on which the result is based.

Let
\[
    \Theta
    :=
    \Theta(L)
    =
    \bigcup_{W\ge 1}
    \bigcup_{S\ge 1}
    \bigcup_{B>0}
    \Theta(L,W,S,B),
\]
where \(L=O(1)\) is fixed. A spike-and-slab prior \(\Pi_\theta\) is placed
on \(\Theta\) as follows. The width \(W\) is assigned the prior
\[
    \Pi_\theta(W=w)
    =
    \frac{\lambda_W^w}{(e^{\lambda_W}-1)w!},
    \qquad w=1,2,\ldots,
\]
for some \(\lambda_W>0\). Conditional on \(W=w\), let \(T=O(w^2)\) be
the total number of parameters and let
\[
    \gamma_i\mid W=w
    \sim
    \operatorname{Ber}\left((1+T\lambda_S)^{-1}\right),
    \qquad i\in[T],
\]
for some \(\lambda_S>0\). The parameter bound is assigned the prior
\[
    B\sim \operatorname{Exp}(\lambda_B),
\]
for some \(\lambda_B>0\), and the individual parameters satisfy
\[
    \theta_i\mid (B=b,\gamma_i=1)\sim U([-b,b]),
    \qquad
    \theta_i\mid \gamma_i=0\sim \delta_0 .
\]
Recall the DNN solution space
\[
    \mathcal F
    :=
    \left\{
        \operatorname{clip}\circ f_\theta:
        \overline{\Omega}\to\mathbb R;
        \theta\in\Theta
    \right\}.
\]
Since the networks use the \(\sigma_3\) activation function, we have
\(\mathcal F\subset H^2(\Omega)\). We equip \(\mathcal F\) with the trace
\(\sigma\)-algebra \(\Sigma_{\mathcal F}\) induced by the Borel
\(\sigma\)-algebra in \(H^2(\Omega)\). The induced prior \(\Pi\) on
\(\mathcal F\) is defined by
\begin{equation}
    \Pi(E)
    :=
    \Pi_\theta
    \left(
        \left\{
            \theta\in\Theta:
            \operatorname{clip}\circ f_\theta\in E
        \right\}
    \right),
    \qquad E\in\Sigma_{\mathcal F}.
    \label{eq:ZL-prior-form}
\end{equation}
\begin{theorem}
Assume that \(u^\ast \in C^\beta(\overline{\Omega})\) is the solution to \eqref{eq:elliptic-model} with \(f \in C^{\beta-2}(\overline{\Omega})\) and \(g \in C^\beta(\partial\Omega)\). Assume that the smoothness level \(\beta\) satisfies $ 2 < \beta \le \beta^\ast$ for some \(\beta^\ast > 2\), and $ \|u^\ast\|_{C^\beta(\overline{\Omega})} \le K$ for some \(K>0\). Then there exists a prior \(\Pi\) of the form \eqref{eq:ZL-prior-form} on a DNN space \(\mathcal F\) such that
\[
    \Pi\left(
        u \in \mathcal F :
        \mathcal E(u) > M_n^2 \varepsilon_n^2
        \,\middle|\, \mathcal D^{(n)}
    \right)
    \to 0,
\]
in \(P_{u^\ast}^{(n)}Q_{u^\ast}^{(n)}\)-probability as \(n\to\infty\)
for any \(M_n\to\infty\).
\end{theorem}
When the number of interior and boundary measurements differ, \cite{ZhaoLu2026} provides a minimax lower bound for the physics-informed loss, $\mathcal E(u)$, defined in \eqref{eloss}. The result is stated as follows.

\begin{theorem}[{\cite[Theorem 3.5]{ZhaoLu2026}}]\label{thm:ZL-minimax}
Assume \(u^*\in C^\beta(\bar{\Omega})\) is a solution to \eqref{eq:elliptic-model} with
\(\beta>2\), and \(\|u^*\|_{C^\beta(\bar{\Omega})}\le K\). For
\(n_\Omega\in\mathbb N\), let \(\{X_i\}_{i=1}^{n_\Omega}\) be an i.i.d.
sequence of random variables distributed according to the uniform distribution
over \(\Omega\), and \(f_i=f(X_i)+\epsilon_i\) with
\(\epsilon_i\stackrel{\mathrm{i.i.d.}}{\sim}\mathcal N(0,1)\) being independent
from \(X_i\). For \(n_{\partial}\in\mathbb N\), let
\(\{Y_j\}_{j=1}^{n_{\partial}}\) be an i.i.d. sequence of uniformly
distributed random variables over \(\partial\Omega\), and
\(g_j=g(Y_j)+\eta_j\) with
\(\eta_j\stackrel{\mathrm{i.i.d.}}{\sim}\mathcal N(0,1)\) being independent
from \(Y_j\). Denote $\mathcal D=\{X_i,f_i\}_{i=1}^{n_\Omega}
\cup\{Y_j,g_j\}_{j=1}^{n_{\partial}}.$ Then we have
\[
\inf_{\psi}\sup_{u^*}\mathbb E\,\mathcal E(\psi(\mathcal D))
\gtrsim
n_\Omega^{-\frac{2(\beta-2)}{d+2(\beta-2)}}
+
n_{\partial}^{-\frac{2\beta}{d-1+2\beta}},
\]
where the supremum is taken over all \(u^*\in C^\beta(\bar{\Omega})\) which
is a solution to \eqref{eq:elliptic-model} with \(f\in C^{\beta-2}(\bar{\Omega})\),
\(g\in C^\beta(\partial \Omega)\) such that
\(\|u^*\|_{C^\beta(\bar{\Omega})}\le K\), and the infimum is taken over all
estimators $\psi:
(\mathbb R^d)^{\otimes n_\Omega}
\times
\mathbb R^{\otimes n_\Omega}
\times
(\mathbb R^d)^{\otimes n_{\partial}}
\times
\mathbb R^{\otimes n_{\partial}}
\to
C^\beta(\bar{\Omega}).$
\end{theorem}

\section{Operator-split Bayesian Posterior}\label{sec:split}

In this section, we construct an operator-split Bayesian posterior for the true solution $u^\ast$ to \eqref{eq:elliptic-model}. Our construction regards the interior source $f$ and boundary data $g$ as separate statistical regression problems, and then obtains the corresponding solution by applying the elliptic solution map.

We first introduce the solution maps associated with these data. Independent Bayesian neural-network (BNN) priors are then assigned to $f$ and to the local-coordinate representation of $g$ on the boundary. The resulting posteriors for $f$ and $g$ are combined as a product posterior and pushed forward through the solution map, yielding a posterior law on the solution space.

\subsection{Solution Operators}

Let \(\Omega\subset\mathbb R^d\) be a bounded domain with smooth boundary, and assume that the coefficients of \(\mathcal L\) satisfy the regularity assumptions stated in Section \eqref{elipsec}. For \(\beta>2\), we define the zero-boundary weak solution operator $\mathcal S_0:L^2(\overline\Omega)\to H_0^1(\Omega)$ by $v=\mathcal S_0F$, where \(v\) is the unique weak solution of
\begin{align}\label{eq:S0-definition}
    \mathcal L v &= F \quad  \text{in } \Omega, \\
    v &= 0 \quad \text{on } \partial \Omega.\notag 
\end{align}
We also define the boundary lifting operator, $\mathcal S_1:H^{1/2}(\partial\Omega)\to H^1(\Omega)$ by \(w=\mathcal S_1G\), where \(w\) is the unique weak solution of
\begin{align}\label{eq:H-definition}
    \mathcal L w &= 0 \quad \text{in } \Omega,\\
    w &= G \quad \text{on } \partial \Omega \notag
\end{align}
\begin{figure}[H]
    \centering
    \includegraphics[scale=0.25]{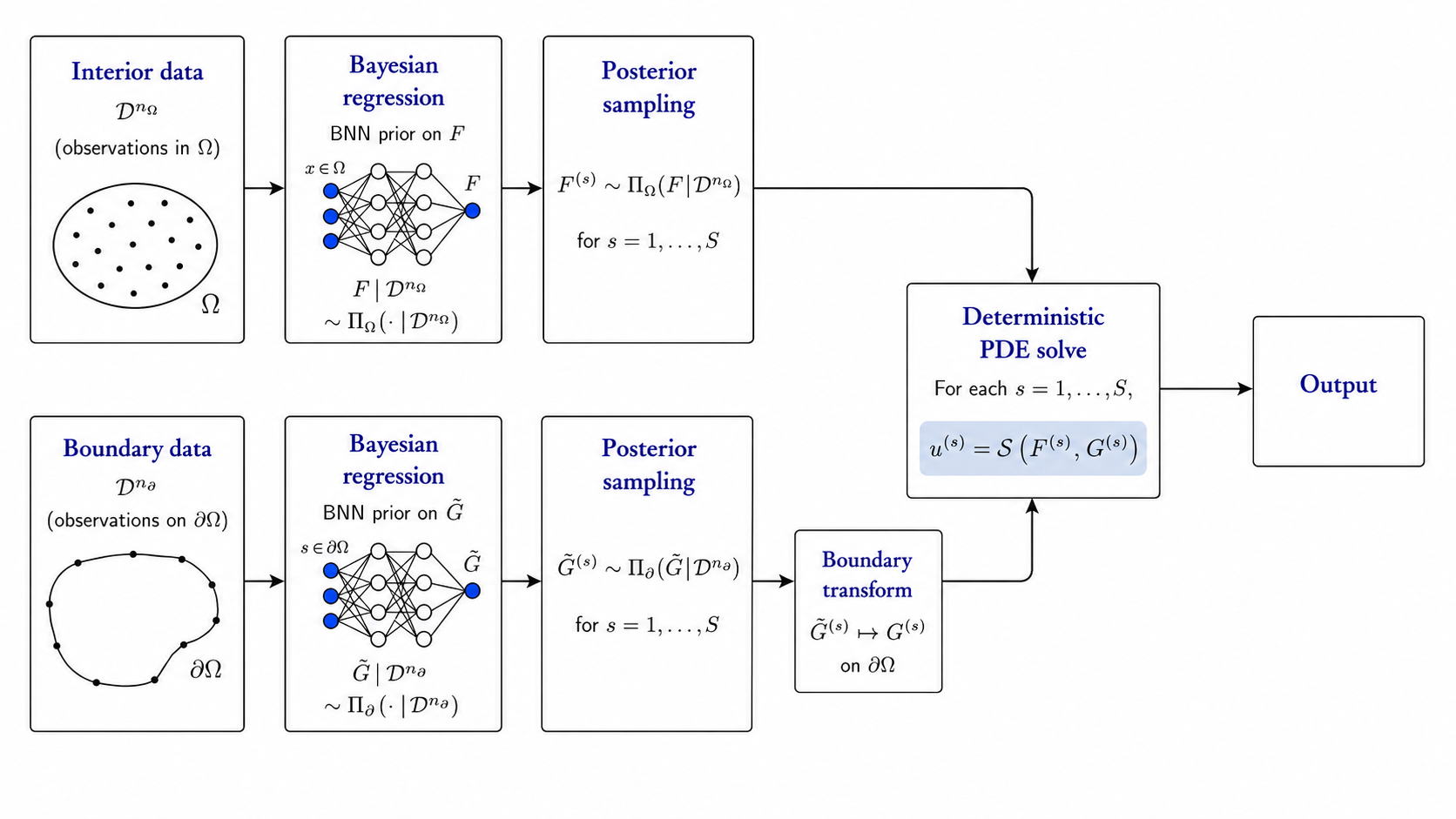}
    \vspace{-2.6em}
    \caption{
    Operator-split Bayesian learning for elliptic problem. Interior observations are used to construct a BNN posterior for the source term \(F\), whereas boundary observations are used to construct a separate posterior for the local-coordinate boundary representation \(\widetilde G\). 
    Posterior samples \(F^{(s)}\) and \(\widetilde G^{(s)}\) are combined by mapping \(\widetilde G^{(s)}\) to the boundary function \(G^{(s)}\), and the corresponding solution sample is obtained through $u^{(s)}=\mathcal S\bigl(F^{(s)},G^{(s)}\bigr).$  The collection of samples \(\{u^{(s)}\}_{s=1}^S\) induces the pushforward posterior on the solution space, from which the posterior mean, variance and credible bands are computed.
    }
    \label{fig:operator-split-bayesian-framework}
\end{figure}
 
By standard elliptic regularity theory
\cite{Grisvard2011,GilbargTrudinger2001}, the problems
\eqref{eq:S0-definition} and \eqref{eq:H-definition} admit unique weak solutions
\(v=\mathcal S_0F\) and \(w=\mathcal S_1G\). Moreover, there exists a constant
\(c=c(\Omega,\mathcal L)>0\) such that
$\norm{\mathcal S_0F}_{H^1(\Omega)}
\le c\norm{F}_{L^2(\Omega)}
\text{ and }
\norm{\mathcal S_1G}_{H^1(\Omega)}
\le c\norm{G}_{H^{1/2}(\partial\Omega)}.$
When \(F\in C^{\beta-2}(\overline\Omega)\) and
\(G\in C^\beta(\partial\Omega)\), the corresponding weak solution is classical,
and the global Schauder estimate gives a constant
\(c=c(\Omega,\mathcal L,\beta)>0\) such that
$\norm{\mathcal S(F,G)}_{C^\beta(\overline\Omega)}
\le
c\left(
\norm{F}_{C^{\beta-2}(\overline\Omega)}
+
\norm{G}_{C^\beta(\partial\Omega)}
\right).$

For any pair $(F,G)\in C^{\beta-2}(\overline\Omega)\times C^\beta(\partial\Omega),$ denote
\begin{equation}
u^{(F,G)} = \mathcal S(F,G):=\mathcal S_0F+\mathcal S_1G .
\label{eq:solution-map}
\end{equation}
Then \(u^{(F,G)}\in C^\beta(\overline\Omega)\), and by linearity, we obtain 
\begin{align}\label{eq:split-identities}
    \mathcal L u^{(F,G)} &= F \quad  \text{in } \Omega, \\
    u^{(F,G)} &= G \quad \text{on } \partial \Omega.\notag 
\end{align}
Consequently, from \eqref{eloss} we get operator-split physics-informed loss
\begin{equation}
\mathcal E\left(u^{(F,G)}\right)
=
\norm{F-f}_{L^2(\Omega)}^2
+
\lambda\norm{G-g}_{L^2(\partial\Omega)}^2 .
\label{eq:split-loss}
\end{equation}

\subsection{Boundary parametrization}
\label{subsec:boundary-parametrization}

We use local coordinates to represent boundary data as data on a subset of
\(\R^{d-1}\). Let \(x^0\in\partial\Omega \in \mathbb R^d\). Since \(\partial\Omega\) is smooth, there exist a neighborhood \(\mathcal V\) of \(x^0\), an open set
\(\mathcal B\subset\mathbb R^{d-1}\), and a smooth chart
\[
\varphi:\mathcal B\to \dOm\cap \mathcal V
\]
which is a diffeomorphism onto its image. Thus, each boundary point \(y\in\dOm\cap \mathcal V\) can be written uniquely as $y=\varphi(z),
\, z\in \mathcal B.$ This is the standard boundary-straightening representation used, for example, in \cite[Appendix~C.1]{Evans2010}. The variable \(z\) gives the local boundary coordinate. We choose the coordinate patch \(\mathcal B\) so that \(\varphi\) extends smoothly to \(\overline{\mathcal B}\) with nondegenerate parametrization.

For a boundary function \(G:\dOm\to\R\), denote its coordinate representation on \(\mathcal B\) by $\widetilde G(z)=G(\varphi(z)), \, z \in \mathcal{B}$. In particular, for the true boundary data \(g\), define $\widetilde g(z)=g(\varphi(z)).$ Thus, a function on the boundary patch \(\dOm\cap \mathcal V\) is represented in local
coordinates as a function on \(\mathcal B\subset\R^{d-1}\).

Let \(\mathcal J_\varphi\) denote the surface Jacobian of the chart \(\varphi\),
\[
\mathcal J_\varphi(z)
=
\sqrt{\det\!\left((D\varphi(z))^\top D\varphi(z)\right)},
\]
where \(D\varphi(z)\) denotes the derivative matrix of \(\varphi\) with respect to the local boundary coordinate \(z\).

\begin{lemma}\label{lem:boundary-l2-coordinates}
Let \(\varphi:\mathcal B\to\dOm\cap \mathcal V\) be a smooth nondegenerate boundary chart, with \(\mathcal{\overline B}\) compact. Then there exist constants \(0<c_\varphi\le C_\varphi<\infty\) such that
\[
c_\varphi
\norm{\widetilde G-\widetilde g}_{L^2(\mathcal B)}^2
\le
\Bigl\|G-g\Bigr\|_{L^2(\dOm\cap \mathcal V)}^2
\le
C_\varphi
\norm{\widetilde G-\widetilde g}_{L^2(\mathcal B)}^2 .
\]
\end{lemma}

\begin{proof}
By the surface change-of-variables formula, it is easy to see that
\[
\Bigl\|{G-g}\Bigr\|_{L^2(\dOm\cap \mathcal V)}^2
=
\int_\mathcal{B}
\abs{\widetilde G(z)-\widetilde g(z)}^2
\mathcal J_\varphi(z)\,\dd z .
\]
Since \(\mathcal{\overline B}\) is compact and \(\varphi\) is smooth and nondegenerate,
\(\mathcal J_\varphi\) is continuous and strictly positive on \(\overline B\). Hence, there
exist constants \(0<c_\varphi\le C_\varphi<\infty\) such that $c_\varphi\le \mathcal J_\varphi(z)\le C_\varphi,\, z\in \mathcal B.$ The result follows immediately.
\end{proof}

\begin{lemma}\label{lem:boundary-observations-coordinates}
Let \(Y\in\dOm\cap \mathcal V\), and suppose that $g_Y=g(Y)+\eta$ is a noisy boundary observation. If \(Y=\varphi(Z)\) for some \(Z\in \mathcal B\), then $g_Y=\widetilde g(Z)+\eta.$ Moreover, if \(Y\sim \mathcal{U}(\dOm\cap\mathcal V)\), then the induced density of \(Z=\varphi^{-1}(Y)\) on \(\mathcal B\) is bounded.
\end{lemma}

\begin{proof}
Since \(Y=\varphi(Z)\), we have
\[
g_Y
=
g(Y)+\eta
=
g(\varphi(Z))+\eta
=
\widetilde g(Z)+\eta.
\]
If \(Y\sim \mathcal{U}(\dOm\cap\mathcal V)\), then the surface change-of-variables formula gives the induced density
of \(Z=\phi^{-1}(Y)\) as
\[
p_Z(z)
=
\frac{\mathcal J_\varphi(z)}
{\int_{\mathcal B} \mathcal J_\varphi(z')\,\dd z'}.
\]
The bounds $0<c_\varphi\le J_\varphi(z)\le C_\varphi<\infty$ imply that
$p_Z$ is bounded above and bounded away from zero on $B$.
\end{proof}

Consequently, on a boundary chart, the boundary observations on
\(\dOm \cap \mathcal V\) are rewritten as observations of the coordinate function
\(\widetilde g:\mathcal B\to\R\), where \(\mathcal B\subset\R^{d-1}\). This is the sense in which the boundary component is represented using its intrinsic \((d-1)\)-dimensional coordinates.

\begin{remark}[Finite-atlas]\label{rem:finite-chart}
A single coordinate chart need not cover the whole boundary. For a compact smooth boundary, finitely many such charts cover
\(\dOm\). Thus, for $\ell=1,\ldots,N,$ one may choose charts
\[
\varphi_\ell:\mathcal B_\ell\subset\R^{d-1}\to\dOm\cap \mathcal V_\ell,
\]
such that $\dOm\subset\bigcup_{\ell=1}^N \mathcal V_\ell.$ Let \(\{\chi_\ell\}_{\ell=1}^L\) be a smooth partition of unity subordinate to
this cover, so that $0\le \chi_\ell\le 1,\, \supp\chi_\ell\subset \mathcal V_\ell,$ and $\sum_{\ell=1}^L \chi_\ell(y)=1$ for $y\in\dOm.$ Then boundary integrals can be written as finite sums of coordinate
integrals. In particular, for boundary functions \(G\) and \(g\), one has
\[
\Bigl\|G-g\Bigr\|_{L^2(\dOm)}^2
=
\sum_{\ell=1}^L
\int_{\mathcal B_\ell}
\chi_\ell(\varphi_\ell(z))
\abs{G(\varphi_\ell(z))-g(\varphi_\ell(z))}^2
\mathcal J_{\varphi_\ell}(z)\,\dd z .
\]
\end{remark}

\subsection{Pushforward Posterior}\label{push}
Given the training data $\mathcal D^{(n_\Omega)} = \{(X_i,f_i)\}_{i=1}^{n_\Omega}$ and $\mathcal D^{(n_\partial)} = \{(Y_j,g_j)\}_{j=1}^{n_\partial}$, denoting the interior and boundary data sets, respectively, we write $\mathcal D^{(n_\Omega,n_\partial)}=\mathcal D^{(n_\Omega)}\cup \mathcal D^{(n_\partial)}.$ Using the local boundary coordinates introduced in the preceding section, with \(Y_j=\varphi(Z_j)\), we represent the boundary data as $\widetilde{\mathcal D}^{(n_\partial)}=\{(Z_j,g_j)\}_{j=1}^{n_\partial}.$

Conditional on \(F\) and the interior
noise variance \(\sigma_\Omega^2\), the interior observations satisfy
\[
f_i\mid X_i,F,\sigma_\Omega^2
\sim
\mathcal N(F(X_i),\sigma_\Omega^2),
\qquad i=1,\ldots,n_\Omega .
\]
Similarly, conditional on \(\widetilde G\) and the boundary noise variance
\(\sigma_\partial^2\), the boundary observations satisfy
\[
g_j\mid Z_j,\widetilde G,\sigma_\partial^2
\sim
\mathcal N(\widetilde G(Z_j),\sigma_\partial^2),
\qquad j=1,\ldots,n_\partial .
\label{eq:boundary-bayesian-model}
\]
Let \(\Pi_\Omega\) and \(\Pi_\partial\) be the BNN priors
for \(F\) and \(\widetilde G\), respectively. Let \(\Xi_\Omega\) and
\(\Xi_\partial\) be the variance priors for \(\sigma_\Omega^2\) and
\(\sigma_\partial^2\), respectively. Thus, the interior prior
is \(\Pi_\Omega\otimes \Xi_\Omega\), and the boundary prior is
\(\Pi_\partial\otimes \Xi_\partial\). The corresponding likelihoods are given by 
\begin{align*}
    \mathscr L_\Omega(F,\sigma_\Omega^2;\mathcal D^{(n_\Omega)})
&=
\prod_{i=1}^{n_\Omega}
\frac{1}{\sqrt{2\pi\sigma_\Omega^2}}
\exp\left[
-\frac{\{f_i-F(X_i)\}^2}{2\sigma_\Omega^2}
\right],\\
\mathscr L_\partial(\widetilde G,\sigma_\partial^2;
\widetilde{\mathcal D}^{(n_\partial)})
&=
\prod_{j=1}^{n_\partial}
\frac{1}{\sqrt{2\pi\sigma_\partial^2}}
\exp\left[
-\frac{\{g_j-\widetilde G(Z_j)\}^2}{2\sigma_\partial^2}
\right].
\end{align*}
By Bayes' formula, the interior and boundary posteriors are
\begin{align*}
    \Pi_\Omega(\dd F,\dd\sigma_\Omega^2\mid\mathcal D^{(n_\Omega)})
 &\propto
\mathscr L_\Omega(F,\sigma_\Omega^2;\mathcal D^{(n_\Omega)})
\Pi_\Omega^0(\dd F)\Xi_\Omega(\dd\sigma_\Omega^2),\\
\Pi_\partial(\dd\widetilde G,\dd\sigma_\partial^2
\mid
\widetilde{\mathcal D}^{(n_\partial)})
&\propto
\mathscr L_\partial(\widetilde G,\sigma_\partial^2;
\widetilde{\mathcal D}^{(n_\partial)})
\Pi_\partial^0(\dd\widetilde G)\Xi_\partial(\dd\sigma_\partial^2).
\end{align*}
Since the two priors are independent and the two likelihoods factor, the joint posterior over all unknown quantities
\((F,\widetilde G,\sigma_\Omega^2,\sigma_\partial^2)\) satisfies
\begin{align}
&\Pi
(\dd F,\dd\widetilde G,\dd\sigma_\Omega^2,\dd\sigma_\partial^2
\mid
\mathcal D^{(n_\Omega,n_\partial)})
\notag\\
&\quad \propto
\mathscr L_\Omega(F,\sigma_\Omega^2;\mathcal D^{(n_\Omega)})
\mathscr L_\partial(\widetilde G,\sigma_\partial^2;
\widetilde{\mathcal D}^{(n_\partial)})
\notag\\
&\quad \times
\Pi_\Omega(\dd F)\Xi_\Omega(\dd\sigma_\Omega^2)
\Pi_\partial(\dd\widetilde G)\Xi_\partial(\dd\sigma_\partial^2) \notag\\
&\quad =
\Pi_\Omega(\dd F,\dd\sigma_\Omega^2
\mid
\mathcal D^{(n_\Omega)})
\otimes
\Pi_\partial(\dd\widetilde G,\dd\sigma_\partial^2
\mid
\widetilde{\mathcal D}^{(n_\partial)}),
\label{eq:joint-posterior-bayes}
\end{align}
where the product posterior in the final equality is obtained after normalization. 

Given \(\widetilde G\), let \(G\) denote the corresponding boundary function on \(\partial\Omega\). The induced posterior over the solution space is the marginal pushforward
\begin{equation}
\Pi_{\mathcal S}(\mathcal W \mid \mathcal D^{(n_\Omega,n_\partial)})
=
\Pi\left(
(F,\widetilde G,\sigma_\Omega^2,\sigma_\partial^2):
\mathcal S(F,G) \in \mathcal W
\,\middle|\,
\mathcal D^{(n_\Omega,n_\partial)}
\right).
\label{eq:pushforward-posterior}
\end{equation}

\section{Main Result}\label{sec:main-results}
In this section, we establish the main posterior contraction result for the operator-split Bayesian construction. The result shows that the induced posterior over solutions contracts around the true solution, with separate contributions from the interior source observations and the boundary trace observations. The interior term reflects the loss of two derivatives under the second-order elliptic operator and is governed by the \(d\)-dimensional sample size \(n_\Omega\), whereas the boundary term is governed by the intrinsic \((d-1)\)-dimensional sample size \(n_\partial\). We then compare the resulting upper bound with the minimax lower bound of \cite{ZhaoLu2026} and derive a sampling-budget condition under which the boundary contribution does not dominate the contraction radius.

Let \(\gamma>2\) and $\lambda \in (0, \infty)$. For the interior and boundary sample sizes \(n_\Omega\) and \(n_\partial\), define
\begin{equation}
\eps_{n_\Omega}
=
n_\Omega^{-\frac{\beta-2}{d+2(\beta-2)}}(\log n_\Omega)^\gamma,
\quad
\eps_{n_\partial}
=
n_\partial^{-\frac{\beta}{d-1+2\beta}}(\log n_\partial)^\gamma,
\label{eq:two-radii}
\end{equation}
and set $R_{n_\Omega,n_\partial}^2=\eps_{n_\Omega}^2+\lambda\eps_{n_\partial}^2.$

The following result is an immediate consequence of Theorem~\ref{thm:KK-regression}.

\begin{corollary}\label{cor:KK-function-contraction}
Under the assumptions of Theorem \ref{thm:KK-regression}, for every
\(M_n\to\infty\),
\begin{equation}
\Pi_n\left(
(\boldsymbol \theta,\sigma^2):
\left\|T_B\circ h_{\boldsymbol \theta,\boldsymbol \rho_\nu}^{\mathrm{DNN}}
-h_0\right\|_{L^2(\mathrm P_X)}
>
M_n\varepsilon_n
\;\middle|\;
\mathcal D^{(n)}
\right)
\to 0
\label{eq:generic-regression-contraction}
\end{equation}
in \(\mathrm P_0^{(n)}\)-probability as \(n\to\infty\).
\end{corollary}

\begin{proof}
By Theorem \ref{thm:KK-regression},
\[
\Pi_n\left(
(\boldsymbol \theta,\sigma^2):
\left\|T_B\circ h_{\boldsymbol \theta,\boldsymbol \rho_\nu}^{\mathrm{DNN}}
-h_0\right\|_{L^2(\mathrm P_X)}
+
|\sigma^2-\sigma_0^2|
>
M_n\varepsilon_n
\;\middle|\;
\mathcal D^{(n)}
\right)
\to 0.
\]
Since
\[
\left\|T_B\circ h_{\boldsymbol \theta,\boldsymbol \rho_\nu}^{\mathrm{DNN}}
-h_0\right\|_{L^2(\mathrm P_X)}
\le
\left\|T_B\circ h_{\boldsymbol \theta,\boldsymbol \rho_\nu}^{\mathrm{DNN}}
-h_0\right\|_{L^2(\mathrm P_X)}
+
|\sigma^2-\sigma_0^2|,
\]
the event in \eqref{eq:generic-regression-contraction} is contained in the event controlled by Theorem \ref{thm:KK-regression}. Hence its posterior
probability also converges to zero.
\end{proof}
We now state our main contraction theorem, which shows that the operator-split posterior concentrates around the true solution \(u^\ast\) at the rate $R_{n_\Omega,n_\partial}^2=\eps_{n_\Omega}^2+\lambda\eps_{n_\partial}^2.$ 
 
\begin{theorem}\label{thm:main-two-sample}
Assume that \(u^\ast\in C^\beta(\overline\Omega)\) is the solution to
\eqref{eq:elliptic-model}, with
\(f\in C^{\beta-2}(\overline\Omega)\) and
\(g\in C^\beta(\partial\Omega)\), where \(\beta>2\), and suppose that $\norm{f}_{C^{\beta-2}(\overline\Omega)}
+
\norm{g}_{C^\beta(\partial\Omega)}
\le K$
for some \(K>0\). Assume further that the boundary chart
$\varphi:\mathcal B\to\partial\Omega\cap \mathcal V$ covers
$\partial\Omega$ up to a set of surface measure zero. Let \(B\ge K\). Let \(\Pi_\Omega\) and \(\Pi_\partial\) be the posteriors
induced by $F=T_B\circ f_{\boldsymbol\theta,\boldsymbol\rho_\nu}^{\mathrm{DNN}}$ and $\widetilde G
=
T_B\circ g_{\boldsymbol\theta,\boldsymbol\rho_\nu}^{\mathrm{DNN}},$ respectively, with architecture \eqref{arch}, and suppose that their parameter and variance priors satisfy Assumptions~\ref{ass:KK-parameter-prior} and~\ref{ass:KK-variance-prior}. Let \(\Pi_{\mathcal S}\) be the operator-split posterior defined in \eqref{eq:pushforward-posterior}. Then, for every \(M_{n_\Omega,n_\partial}\to\infty\),
\begin{equation}
\Pi_{\mathcal S}\left(
\left\{
u:
\calE(u)
>
M_{n_\Omega,n_\partial}^2
R_{n_\Omega,n_\partial}^2
\right\}
\,\middle|\,
\mathcal D^{(n_\Omega,n_\partial)}
\right)
\to 0
\label{eq:main-contraction-statement}
\end{equation}
in \(\mathrm{P}_{u^\ast}^{(n_\Omega)}
\mathrm{Q}_{u^\ast}^{(n_\partial)}\)-probability as
\(n_\Omega\wedge n_\partial\to\infty\).
\end{theorem}

\begin{proof}
We apply Corollary~\ref{cor:KK-function-contraction} to the interior regression problem with
\(X_i\sim \mathcal U(\Omega)\), \(f_i=f(X_i)+\eps_i\),
\(\eps_i\stackrel{\mathrm{i.i.d.}}{\sim}\mathcal N(0,\sigma_{\Omega,0}^2)\),
\(\alpha=\beta-2\), and \(n=n_\Omega\). Since \(f\in C^{\beta-2}(\overline\Omega)\) and
\(\norm{f}_{C^{\beta-2}(\overline\Omega)}\le K\), the assumptions of Corollary~\ref{cor:KK-function-contraction} are satisfied for the interior problem.
Hence, for every \(M_{n_\Omega,n_\partial}\to\infty\), we have 
\begin{align*}
\Pi_\Omega\left(
\left\{
F:
\norm{F-f}_{L^2(\mathrm{P}_X)}
>
M_{n_\Omega,n_\partial}\eps_{n_\Omega}
\right\}
\,\middle|\,
\mathcal D^{(n_\Omega)}
\right)
\to 0
\label{eq:interior-contraction-proof}
\end{align*}
in \(\mathrm{P}_{u^\ast}^{(n_\Omega)}\)-probability. Since \(X_i\sim \mathcal U(\Omega)\), the \(L^2(\mathrm{P}_X)\) and \(L^2(\Omega)\) norms differ only by the constant \(|\Omega|^{-1/2}\). Absorbing this fixed constant into the diverging sequence \(M_{n_\Omega,n_\partial}\), we obtain
\begin{equation}
\Pi_\Omega\left(
\left\{
F:
\norm{F-f}_{L^2(\Omega)}
>
M_{n_\Omega,n_\partial}\eps_{\Omega,n_\Omega}
\right\}
\,\middle|\,
\mathcal D_\Omega^{(n_\Omega)}
\right)
\to 0 .
\label{eq:interior-l2-omega-proof}
\end{equation}

Next, we apply  Corollary~\ref{cor:KK-function-contraction} to the boundary-coordinate regression problem. Write
\(Y_j=\varphi(Z_j)\) and \(\widetilde g(z)=g(\varphi(z))\). For $j= 1,\cdots, n_\partial$, the boundary observations become $g_j=\widetilde g(Z_j)+\eta_j$. Since \(g\in C^\beta(\partial\Omega)\) and the boundary chart $\varphi$ is smooth,
\(\widetilde g\in C^\beta(\mathcal B)\). Therefore, for every \(M_{n_\Omega,n_\partial}\to\infty\),
\begin{align*}
\Pi_\partial\left(
\left\{
\widetilde G:
\norm{\widetilde G-\widetilde g}_{L^2(\mathrm{P}_Z)}
>
M_{n_\Omega,n_\partial}\eps_{n_\partial}
\right\}
\,\middle|\,
\widetilde{\mathcal D}^{(n_\partial)}
\right)
\to 0
\label{eq:boundary-coordinate-contraction-proof}
\end{align*}
in \(\mathrm{Q}_{u^\ast}^{(n_\partial)}\)-probability. Since the chart $\varphi$ covers $\partial\Omega$ up to a set of surface
measure zero, Lemmas~\ref{lem:boundary-l2-coordinates} and ~\ref{lem:boundary-observations-coordinates} imply that the
$L^2(P_Z)$ and $L^2(\partial\Omega)$ norms are equivalent up to fixed multiplicative constants. Absorbing these constants into \(M_{n_\Omega,n_\partial}\), we get
\begin{equation}
\Pi_\partial\left(
\left\{
G:
\norm{G-g}_{L^2(\partial\Omega)}
>
M_{n_\Omega,n_\partial}\eps_{n_\partial}
\right\}
\,\middle|\,
\widetilde{\mathcal D}^{(n_\partial)}
\right)
\to 0 .
\label{eq:boundary-l2-proof}
\end{equation}
By the operator-split identity \eqref{eq:split-loss}, every posterior draw \(u^{(F,G)}=\mathcal S(F,G)\) satisfies
\[
\calE(u^{(F,G)})
=
\norm{F-f}_{L^2(\Omega)}^2
+
\lambda\norm{G-g}_{L^2(\partial\Omega)}^2 .
\]
Thus, if $\norm{F-f}_{L^2(\Omega)} \le
M_{n_\Omega,n_\partial}\eps_{n_\Omega}$ and $\norm{G-g}_{L^2(\partial\Omega)} \le M_{n_\Omega,n_\partial}\eps_{n_\partial},$ then $\calE(u^{(F,G)})
\le
M_{n_\Omega,n_\partial}^2R_{n_\Omega,n_\partial}^2.$ Consequently,
\begin{align*}
&\left\{
u^{(F,G)}:
\calE(u^{(F,G)})
>
M_{n_\Omega,n_\partial}^2R_{n_\Omega,n_\partial}^2
\right\} \\
&\qquad \subset \left\{
F:
\norm{F-f}_{L^2(\Omega)}
>
M_{n_\Omega,n_\partial}\eps_{n_\Omega}
\right\} \cup \left\{
G:
\norm{G-g}_{L^2(\partial\Omega)}
>
M_{n_\Omega,n_\partial}\eps_{n_\partial}
\right\}.
\end{align*}
Using the pushforward definition of \(\Pi_{\mathcal S}\) and the product-posterior factorization, we obtain
\begin{align}
&\Pi_{\mathcal S}\left(
\left\{
u:
\calE(u)
>
M_{n_\Omega,n_\partial}^2R_{n_\Omega,n_\partial}^2
\right\}
\,\middle|\,
\mathcal D^{(n_\Omega,n_\partial)}
\right)
\notag\\
&\qquad\le
\Pi_\Omega\left(
\left\{
F:
\norm{F-f}_{L^2(\Omega)}
>
M_{n_\Omega,n_\partial}\eps_{n_\Omega}
\right\}
\,\middle|\,
\mathcal D^{(n_\Omega)}
\right)
\notag\\
&\qquad\quad+
\Pi_\partial\left(
\left\{
G:
\norm{G-g}_{L^2(\partial\Omega)}
>
M_{n_\Omega,n_\partial}\eps_{n_\partial}
\right\}
\,\middle|\,
\widetilde{\mathcal D}^{(n_\partial)}
\right).
\label{eq:main-proof-union-bound}
\end{align}
Combining \eqref{eq:interior-l2-omega-proof}, \eqref{eq:boundary-l2-proof}, and \eqref{eq:main-proof-union-bound} we obtain
\[
\Pi_{\mathcal S}\left(
\left\{
u:
\calE(u)
>
M_{n_\Omega,n_\partial}^2R_{n_\Omega,n_\partial}^2
\right\}
\,\middle|\,
\mathcal D^{(n_\Omega,n_\partial)}
\right)
\to 0
\]
in \(\mathrm{P}_{u^\ast}^{(n_\Omega)}\mathrm{Q}_{u^\ast}^{(n_\partial)}\)-probability as
\(n_\Omega\wedge n_\partial\to\infty\). 
\end{proof}

\begin{remark}[Contraction of the posterior mean]\label{postmean}
Let 
$
    \bar F=\E_{\Pi_\Omega}\left[F\mid \mathcal D^{(n_\Omega)}\right]\,\text{and} \,\bar G = \E_{\Pi_\partial}\left[G\mid \widetilde{\mathcal D}^{(n_\partial)}\right],
$
and denote the posterior mean by $\bar u=\mathcal S(\bar F,\bar G).$ If, in addition, the second-moment contractions hold, i.e,
\[
\E_{\Pi_\Omega}\left[
\norm{F-f}_{L^2(\Omega)}^2
\mid \mathcal D^{(n_\Omega)}
\right]
=
O_p(\eps_{n_\Omega}^2)
\quad 
\text{and}
\quad 
\E_{\Pi_\partial}\left[
\norm{G-g}_{L^2(\partial\Omega)}^2
\mid \widetilde{\mathcal D}^{(n_\partial)}
\right]
=
O_p(\eps_{n_\partial}^2),
\]
then by linearity of $\mathcal S$, and Jensen's inequality, $\calE(\bar u)=O_p(R_{n_\Omega,n_\partial}^2).$ 
\end{remark}

\begin{remark}[\(H^{1/2}(\Omega)\)-contraction]\label{h-cont}
The contraction statement in Theorem~\ref{thm:main-two-sample} is formulated in the physics-informed loss \(\mathcal E(u)\). By the stability estimate in \cite[Theorem~8]{MullerZeinhofer2022}, this loss also controls the \(H^{1/2}(\Omega)\)-error. Indeed, for any admissible \(u\), the stability estimate and the definition
of \(\mathcal E(u)\) in \eqref{eploss1} give
\begin{align*}
\norm{u-u^\ast}_{H^{1/2}(\Omega)}
&\le C\left(
\norm{\mathcal L(u-u^\ast)}_{L^2(\Omega)}
+
\norm{u-u^\ast}_{L^2(\partial\Omega)}
\right) \\
&\le C\left(1+\lambda^{-1/2}\right)\mathcal E(u)^{1/2},
\end{align*}
for $C>0$. Therefore, Theorem~\ref{thm:main-two-sample} implies that, for every
\(M_{n_\Omega,n_\partial}\to\infty\),
\[
\Pi_{\mathcal S}\left(
\left\{
u:
\norm{u-u^\ast}_{H^{1/2}(\Omega)}
>
M_{n_\Omega,n_\partial}R_{n_\Omega,n_\partial}
\right\}
\,\middle|\,
\mathcal D^{(n_\Omega,n_\partial)}
\right)
\to 0
\]
in \(\mathrm{P}_{u^\ast}^{(n_\Omega)}\mathrm{Q}_{u^\ast}^{(n_\partial)}\)-probability as
\(n_\Omega\wedge n_\partial\to\infty\). Since \(\Omega\) is bounded, this also
implies \(L^2(\Omega)\)-contraction. As in \cite[Remark~3.3]{ZhaoLu2026}, the
\(L^2(\partial\Omega)\) boundary penalty is too weak, by itself, to imply
contraction in \(H^s(\Omega)\) for any \(s>1/2\). Such stronger Sobolev contraction would require either a stronger boundary norm or a separate stability estimate showing that the current loss controls the desired \(H^s(\Omega)\)-error
\end{remark}

\begin{corollary}[Near-minimax contraction rate]\label{cor:near-minimax}
Under the assumptions of Theorem~\ref{thm:main-two-sample}, the contraction radius of the operator-split posterior satisfies
\begin{align}\label{ratepos}
R_{n_\Omega,n_\partial}^2
=
n_\Omega^{-\frac{2(\beta-2)}{d+2(\beta-2)}}(\log n_\Omega)^{2\gamma}
+
\lambda
n_\partial^{-\frac{2\beta}{d-1+2\beta}}(\log n_\partial)^{2\gamma}.
\end{align}
Hence, up to logarithmic factors, the contraction rate matches the two-sample minimax lower bound in Theorem~\ref{thm:ZL-minimax}.
\end{corollary}

\begin{proof}
By definition, $R_{n_\Omega,n_\partial}^2
=
\varepsilon_{n_\Omega}^2
+
\lambda \varepsilon_{n_\partial}^2 .$ Using \eqref{eq:two-radii}, we obtain
\[
R_{n_\Omega,n_\partial}^2
=
n_\Omega^{-\frac{2(\beta-2)}{d+2(\beta-2)}}(\log n_\Omega)^{2\gamma}
+
\lambda
n_\partial^{-\frac{2\beta}{d-1+2\beta}}(\log n_\partial)^{2\gamma}.
\]
On the other hand, Theorem~\ref{thm:ZL-minimax} gives the lower bound
\[
\inf_{\psi}\sup_{u^\ast}
\mathbb E_{u^\ast}
\mathcal E\!\left(\psi(\mathcal D^{(n_\Omega,n_\partial)})\right)
\gtrsim
n_\Omega^{-\frac{2(\beta-2)}{d+2(\beta-2)}}
+
n_\partial^{-\frac{2\beta}{d-1+2\beta}}.
\]
Since \(\lambda>0\) is fixed, the boundary weight only changes the constants
in the comparison. Therefore, the contraction radius in
Theorem~\ref{thm:main-two-sample} agrees with the two-sample minimax lower bound up to logarithmic factors.
\end{proof}

\begin{corollary}[Boundary sampling budget]
\label{cor:boundary-budget}
Ignoring logarithmic factors in \eqref{ratepos}, the boundary term does not dominate the interior term whenever
\begin{equation}
n_\partial\gtrsim n_\Omega^{\kappa(d,\beta)},
\quad
\kappa(d,\beta)=
\frac{(\beta-2)(d-1+2\beta)}{\beta(d+2\beta-4)}.
\label{eq:budget}
\end{equation}
In this regime, the posterior contraction radius is governed by the interior source observations.
\end{corollary}

\begin{proof}
The boundary term is no larger than the interior term, up to constants and logarithmic factors, when
\[
n_\partial^{-\frac{2\beta}{d-1+2\beta}}
\lesssim
n_\Omega^{-\frac{2(\beta-2)}{d+2\beta-4}}.
\]
Solving for $n_\partial$ gives
\[
n_\partial\gtrsim
n_\Omega^{\frac{(\beta-2)(d-1+2\beta)}{\beta(d+2\beta-4)}},
\]
as desired.
\end{proof}

\section{Numerical Experiments}
\label{sec:numerical-experiments}

In this section, we present numerical experiments that illustrate the behavior of the proposed operator-split Bayesian approximation. The first experiment considers a one-dimensional Poisson problem and evaluates the posterior propagation of source and boundary uncertainty through the finite element solver. The second experiment considers a manufactured two-dimensional elliptic problem, examines how the interior and boundary sample sizes affect the posterior reconstruction, and evaluates how the computational time scales with the number of retained posterior samples.

\subsection{1D Poisson Equation}
We consider the one-dimensional Poisson problem with similar formulation as in \cite{YangMengKarniadakis2021}. Let $\Omega=(-0.7,0.7).$ The boundary value problem is
\begin{align}
\alpha u_{xx} &= f
\quad \text{in } \Omega, \nonumber\\
u(-0.7)&=g_L,\nonumber \\ 
u(0.7)&=g_R,
\label{eq_yang_1d_poisson}
\end{align}
where \(\alpha=0.01\). We use the manufactured solution \(u^\ast(x)=\sin^3(6x)\), which gives the source term \(f^\ast(x)=0.01[-27\sin(6x)+81\sin(18x)]\) and the exact boundary data \(g_L=u^\ast(-0.7)\) and \(g_R=u^\ast(0.7)\). We make the assumption that the exact source function and boundary values are unknown. The source data are observed at \(n_\Omega=16\) equidistant sensor locations in \([-0.7,0.7]\), while the boundary data are observed at the two endpoints \(x=-0.7\) and \(x=0.7\). All measurements are assumed to be noisy. We consider two noise levels, $\varepsilon_\Omega,\varepsilon_\partial \sim \mathcal N(0,0.01^2)$ and $\varepsilon_\Omega,\varepsilon_\partial \sim \mathcal N(0,0.1^2),$ for the source and boundary observations, respectively.

\begin{figure}[H]
    \centering
    \includegraphics[width=.8\textwidth]{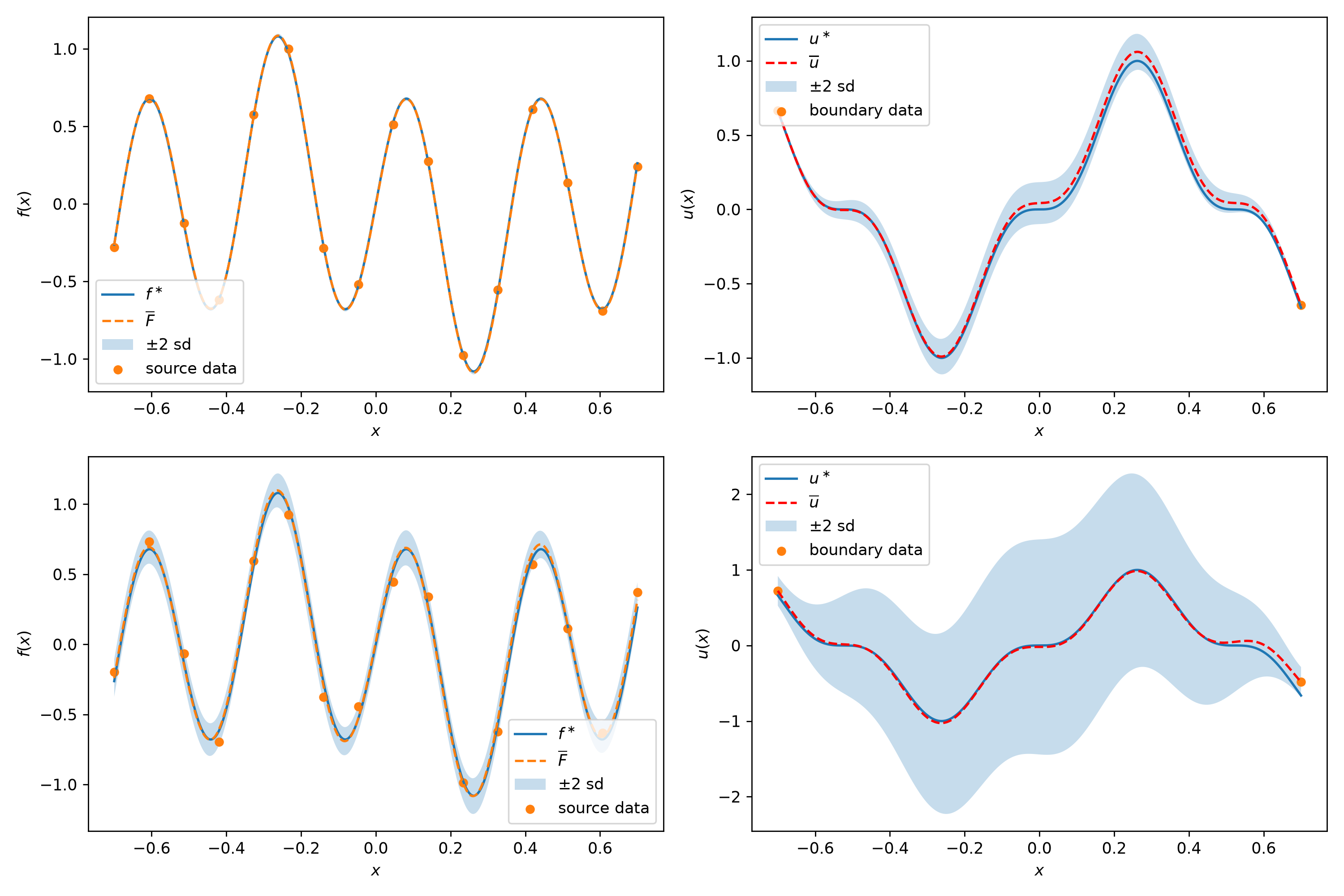}
    \caption{Posterior approximation for the 1D Poisson problem. The source posterior is shown in the left column and the induced solution posterior is shown in the right column, with \(\sigma=0.01\) on the top row and \(\sigma=0.1\) on the bottom row. The dashed curves are posterior means, the shaded regions are two-standard-deviation bands, and the markers are noisy observations.}
    \label{fig:yang-1d-poisson-reconstruction}
\end{figure}
For the operator-split Bayesian approximation, the source term is represented by a Bayesian regression model \(F_{\boldsymbol\theta}\). Since the exact source contains the oscillatory modes \(\sin(6x)\) and \(\sin(18x)\), we use the Fourier feature map
\[
\Phi(x)
=
\left(x,\sin(6x),\cos(6x),\sin(18x),\cos(18x)\right),
\]
and define $F_{\boldsymbol\theta}(x) = \text{NN}_{\boldsymbol\theta}(\Phi(x)).$ The network parameters are assigned the Gaussian prior $\boldsymbol\theta\sim \mathcal N(0,5 I)$. The source posterior is sampled using Hamiltonian Monte Carlo (HMC). The chain is initialized by \(1000\) Adam steps, followed by \(500\) burn-in steps. We use HMC step size \(5\times 10^{-4}\), \(20\) leapfrog steps, thinning interval \(2\), and retain \(500\) posterior samples. Since the one-dimensional boundary consists only of the two endpoints, we sample the boundary values directly instead of fitting a boundary neural network. With the Gaussian prior $G_L,G_R\sim \mathcal N(0,10^2),$ the posterior of each boundary value is Gaussian. For each posterior draw \(\left(F^{(s)},G_L^{(s)},G_R^{(s)}\right)\), we solve the corresponding boundary value problem using piecewise linear finite elements on a uniform mesh with \(512\) subintervals and Gauss--Legendre quadrature of order \(5\). This gives the posterior solution ensemble \(\{u^{(s)}\}_{s=1}^S\), from which we compute the posterior mean and the associated pointwise uncertainty bands.

\subsection{2D Diffusion-Reaction System}
Let $\Omega = (0,1)^2$. We consider the boundary value problem given by
\begin{align}
-\Delta u + u &= f
\quad \text{in } \Omega, \nonumber \\
u&=g \quad \text{on } \partial\Omega.
\label{eq_reaction_diffusion_2d}
\end{align}
We take $u^\ast(x,y)=\sin(\pi x)\sin(\pi y)+x+y$ as the exact solution to \eqref{eq_reaction_diffusion_2d}. Then the source and boundary terms $f^*$ and $g^*=u^\ast|_{\partial\Omega}$ can easily be obtained from \eqref{eq_reaction_diffusion_2d}, which are given by $f^{*}(x,y)=(2\pi^2+1)\sin(\pi x)\sin(\pi y)+x+y$ and 
\[
g^{*}(x,y)=
\begin{cases}
y, & x=0,\\
1+y, & x=1,\\
x, & y=0,\\
x+1, & y=1.
\end{cases}
\]
respectively. We assume that the exact functions \(f\) and \(g\) are not observed continuously. The source data are observed at \(n_\Omega\) uniformly sampled sensor locations in \(\Omega\), whereas the boundary data are observed at \(n_\partial\) uniformly sampled sensor locations on \(\partial\Omega\). We assume that all measurements from the sensors are noisy, and we take \(\varepsilon_{\Omega}\sim \mathcal N(0,0.05^2)\) and \(\varepsilon_{\partial}\sim \mathcal N(0,0.05^2)\) for the source and boundary observations, respectively. 

For the Bayesian regression step, we use two separate fully connected BNNs. The source network \(F_{\boldsymbol \theta}\) maps the spatial coordinate \((x,y)\) to an approximation of the source term \(f(x,y)\). This network has \(4\) hidden layers, each containing \(128\) neurons. The boundary network \(\widetilde G_{\boldsymbol\theta}\) takes the scaled boundary coordinate \(t/4\) as input and approximates \(\tilde g(\gamma(t))\), where \(\gamma:[0,4]\to\partial\Omega\) parametrizes the boundary of the unit square. This network has two hidden layers, each containing \(32\) neurons. Both networks use the activation function $\sigma(\mathbf{z})=\max\{\mathbf{z},0.01\mathbf{z}\}$ and the network parameters are assigned the Gaussian prior $\boldsymbol \theta\sim \mathcal N(0,\tau^2 I),\, \tau=1.$

The posterior samples are generated using stochastic gradient Langevin dynamics (SGLD). For the source network, we use \(10000\) burn-in steps, step size \(5\times 10^{-7}\), batch size \(128\), and retain \(200\) posterior samples after thinning every \(50\) iterations. For the boundary network, we use \(5000\) burn-in steps, step size \(10^{-6}\), batch size \(128\), and retain \(200\) posterior samples with the same thinning interval. For each posterior draw \((F^{(s)},G^{(s)})\), we solve the corresponding boundary value problem using quadratic FEM on a uniform triangular mesh with \(64\) subdivisions in each coordinate direction. The resulting ensemble of FEM solutions is then used to compute the posterior mean and associated uncertainty bands.

\subsubsection{Model Evaluation}
In this section, we detail our evaluation procedure. Let \(S\) denote the number of posterior samples. For each posterior draw
\((F^{(s)},\widetilde G^{(s)})\), we define the corresponding boundary function by
\begin{equation}
G^{(s)}(\gamma(t))=\widetilde G^{(s)}(t/4),
\qquad 0\le t<4.
\end{equation}
The numerical solution obtained from this draw is denoted by $u_h^{(s)}=\mathcal S_h(F^{(s)},G^{(s)}),$ where \(\mathcal S_h\) is the numerical operator for the boundary value problem.

The source error is evaluated on a uniform Cartesian grid with \(N_x\) points in each coordinate direction. For $0\le i,j\le N_x-1,$ let $x_i=i/(N_x-1),\, y_j=j/(N_x-1).$ We define the discrete source error as
\begin{equation}
\mathcal E_F^{(s)}
= \|F^{(s)}-f\|_{L^2(\Omega)}^2 \approx
\Delta x\,\Delta y
\sum_{i=0}^{N_x-1}
\sum_{j=0}^{N_x-1}
\left|
F^{(s)}(x_i,y_j)-f(x_i,y_j)
\right|^2.
\end{equation}
where $\Delta x=1/(N_x-1)$ and $\Delta y=1/(N_x-1)$. 

The boundary error is evaluated on a uniform grid in the boundary coordinate
\(t\). Let \(N_{\partial}\) denote the number of boundary
evaluation points. We write $t_k=4k/N_{\partial},\, 0\le k\le N_{\partial}-1.$ The discrete boundary error for posterior sample \(s\) is defined as
\begin{equation}
\mathcal E_G^{(s)}
=
\|G^{(s)}-g\|_{L^2(\partial\Omega)}^2
\approx
\Delta t
\sum_{k=0}^{N_{\partial}-1}
\left|
G^{(s)}(\gamma(t_k))-g(\gamma(t_k))
\right|^2,
\end{equation}
where $\Delta t=4/N_{\partial}.$

For each posterior draw, we define the discrete analogue of \eqref{eq:split-loss} as 
\begin{equation}\label{split-loss}
\mathcal E(u^{(F, G)})^{(s)}
=
\mathcal E_F^{(s)}
+
\lambda \mathcal E_G^{(s)} .
\end{equation}
For any posterior quantity \(Q^{(s)}\), we define its posterior sample mean and
sample standard deviation by
\begin{align}
\overline Q
&=
\frac{1}{S}\sum_{s=1}^S Q^{(s)},
\qquad
\sigma_Q
=
\left[
\frac{1}{S-1}
\sum_{s=1}^S
\left(Q^{(s)}-\overline Q\right)^2
\right]^{1/2}.
\end{align}

\subsubsection{Numerical Result}

We present the numerical results for the 2D linear diffusion-reaction experiment. As an illustration, we first show the sampling layout for a fixed choice of interior and boundary observations. We then compare the learned source, boundary data, and solution with their exact values. Finally, we examine how the errors change as the number of interior or boundary observations is varied and report the computational cost associated with increasing the posterior ensemble size.

\begin{figure}[H]
    \centering
    \includegraphics[width=0.5\textwidth]{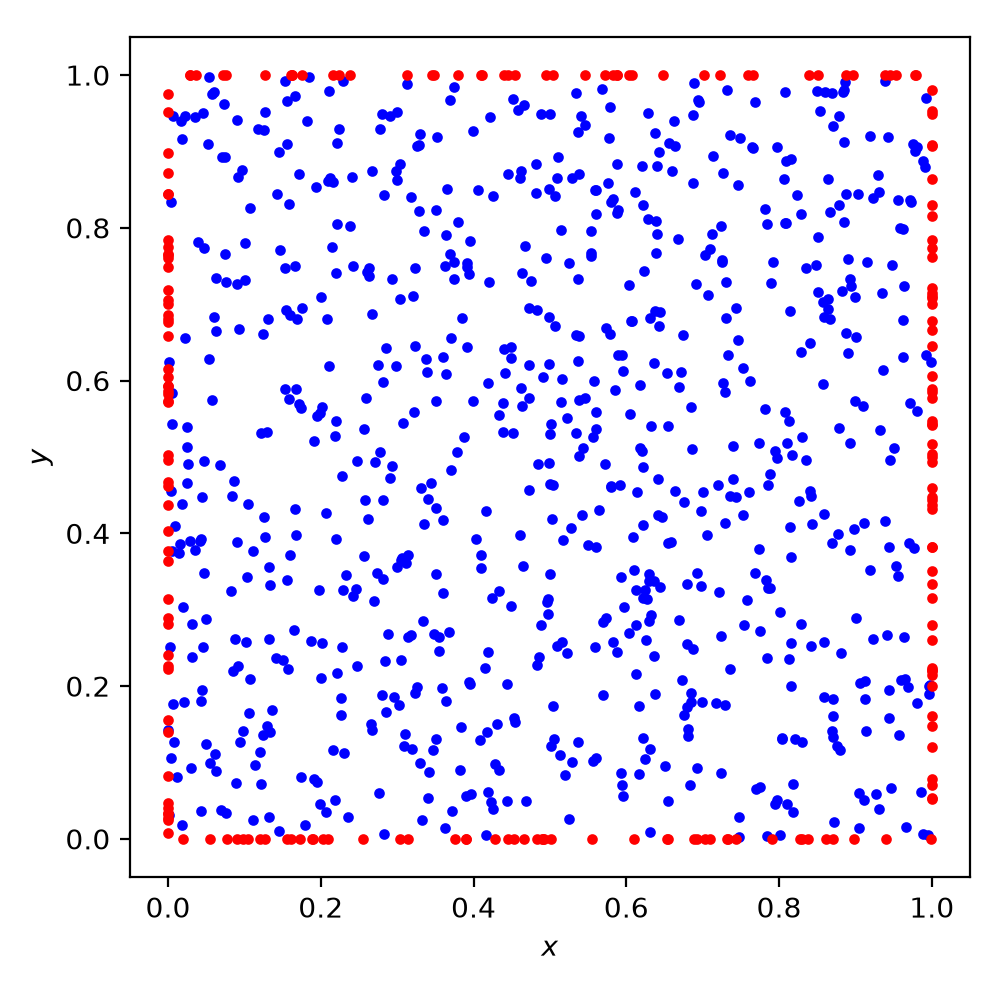}
    \caption{Interior and boundary sampling. Blue points denote interior observations in $\Omega$. Red points denote boundary observations on $\partial\Omega$.}
    \label{fig:single-run-data-layout}
\end{figure}
Figure~\ref{fig:single-run-data-layout} illustrates the sampling configuration with \(n_\Omega=800\) interior observations and \(n_\partial=200\) boundary observations. The interior observation points are sampled uniformly over \(\Omega\), while the boundary observation points are sampled uniformly along \(\partial\Omega\). The operator-split formulation uses these as two distinct data sources, one in the interior of the domain and one on the boundary, with broad spatial coverage of \(\Omega\) and all four boundary segments. These observations form the data used to construct the posterior reconstructions reported in the subsequent results.

\begin{table}[H]
\centering
\small
\caption{Boundary sample-size sweep with fixed interior sample size \(n_\Omega=400,\, \lambda=1\).}
\label{tab:boundary-sweep}
\vspace{.5em}
\begin{tabular}{ccccccc}
\toprule
\(n_\partial\) 
& \(\overline{\mathcal E}_F\)
& \(\sigma_{\mathcal E_F}\)
& \(\overline{\mathcal E}_G\)
& \(\sigma_{\mathcal E_G}\)
& \(\overline{\mathcal E}(u^{(F,G)})\)
& \(\|u^\ast-\overline u\|_{L^2(\Omega)}\) \\
\midrule
25  
& \(3.7537\times 10^{-2}\) 
& \(4.3957\times 10^{-2}\) 
& \(2.4955\times 10^{-1}\) 
& \(3.0777\times 10^{-1}\) 
& \(2.8708\times 10^{-1}\) 
& \(1.6087\times 10^{-2}\) \\

50  
& \(3.7537\times 10^{-2}\) 
& \(4.3957\times 10^{-2}\) 
& \(7.9959\times 10^{-3}\) 
& \(2.6896\times 10^{-3}\) 
& \(4.5533\times 10^{-2}\) 
& \(4.9915\times 10^{-4}\) \\

100 
& \(3.7537\times 10^{-2}\) 
& \(4.3957\times 10^{-2}\) 
& \(2.1275\times 10^{-3}\) 
& \(8.6511\times 10^{-4}\) 
& \(3.9665\times 10^{-2}\) 
& \(8.7000\times 10^{-5}\) \\

200 
& \(3.7537\times 10^{-2}\) 
& \(4.3957\times 10^{-2}\) 
& \(1.1143\times 10^{-3}\) 
& \(5.7318\times 10^{-4}\) 
& \(3.8652\times 10^{-2}\) 
& \(4.9961\times 10^{-5}\) \\
\bottomrule
\end{tabular}
\end{table}

Tables~\ref{tab:boundary-sweep} and~\ref{tab:interior-sweep} depict the separate effects of the boundary and interior sample sizes. In Table~\ref{tab:boundary-sweep}, increasing \(n_\partial\) produces a clear decrease in both the posterior mean and posterior standard deviation of the boundary error, while the source error remains fixed because \(n_\Omega\) is unchanged. For the smallest boundary budget, \(n_\partial=25\), the boundary contribution is larger than the source contribution and dominates the operator-split loss. Once \(n_\partial\) is increased to \(50\), the boundary error falls below the source error, and further increases in \(n_\partial\) yield smaller changes in the total loss. This transition is consistent with Corollary~\ref{cor:boundary-budget}, which predicts that sufficiently many boundary observations prevent the boundary term from dominating the interior term.

In Table~\ref{tab:interior-sweep}, increasing \(n_\Omega\) reduces the posterior mean source error, whereas the boundary error remains essentially unchanged because \(n_\partial\) is fixed. The reduction is most visible from \(n_\Omega=100\) to \(n_\Omega=400\), after which the improvement begins to level off. This behavior is consistent with the rate structure in Corollary~\ref{cor:near-minimax}, where the contraction radius contains separate interior and boundary contributions that decrease with \(n_\Omega\) and \(n_\partial\), respectively. Across both sweeps, the solution error remains small relative to the separate data errors, in agreement with the stability discussion in Remark~\ref{h-cont}. This reflects the smoothing effect of the elliptic solution operator, through which noticeable errors in the learned source and boundary data may still produce a much smaller error in the computed solution.

\begin{table}[H]
\centering
\small
\caption{Interior sample-size sweep with fixed boundary sample size \(n_\partial=100,\, \lambda=1\).}
\label{tab:interior-sweep}
\vspace{.5em}
\begin{tabular}{ccccccc}
\toprule
\(n_\Omega\) 
& \(\overline{\mathcal E}_F\)
& \(\sigma_{\mathcal E_F}\)
& \(\overline{\mathcal E}_G\)
& \(\sigma_{\mathcal E_G}\)
& \(\overline{\mathcal E}(u^{(F,G)})\)
& \(\|u^\ast-\overline u\|_{L^2(\Omega)}\) \\
\midrule
100 
& \(9.7344\times 10^{-2}\) 
& \(4.0232\times 10^{-2}\) 
& \(2.1275\times 10^{-3}\) 
& \(8.6511\times 10^{-4}\) 
& \(9.9471\times 10^{-2}\) 
& \(1.3156\times 10^{-4}\) \\

200 
& \(5.4351\times 10^{-2}\) 
& \(3.0326\times 10^{-2}\) 
& \(2.1274\times 10^{-3}\) 
& \(8.6513\times 10^{-4}\) 
& \(5.6478\times 10^{-2}\) 
& \(9.2686\times 10^{-5}\) \\

400 
& \(3.7537\times 10^{-2}\) 
& \(4.3957\times 10^{-2}\) 
& \(2.1275\times 10^{-3}\) 
& \(8.6511\times 10^{-4}\) 
& \(3.9665\times 10^{-2}\) 
& \(8.7000\times 10^{-5}\) \\

800 
& \(3.6882\times 10^{-2}\) 
& \(5.5538\times 10^{-2}\) 
& \(2.1275\times 10^{-3}\) 
& \(8.6511\times 10^{-4}\) 
& \(3.9010\times 10^{-2}\) 
& \(8.6652\times 10^{-5}\) \\
\bottomrule
\end{tabular}
\end{table}

In Figure~\ref{fig:single-run-boundary-trace}, the posterior mean \(\overline{\widetilde G}(s)\) agrees well with the exact boundary data \(\widetilde g^*(s)\) over the full boundary parametrization. The boundary data are piecewise linear, increasing on the first half of the parameter interval and decreasing on the second half. The posterior mean captures this structure on both segments and remains close to the target values. The largest visible discrepancy occurs near \(s=2\), where the posterior mean slightly smooths the corner of the target function. Away from this point, the two curves are nearly indistinguishable. The pointwise \(95\%\) credible band, computed from the empirical \(2.5\%\) and \(97.5\%\) quantiles of the posterior samples \(\{\widetilde G^{(s)}(s)\}_{s=1}^S\), is narrow throughout the interval. This indicates that the posterior samples for \(\widetilde G\) are tightly concentrated around the posterior mean.
\begin{figure}[H]
    \centering
    \includegraphics[width=0.5\textwidth]{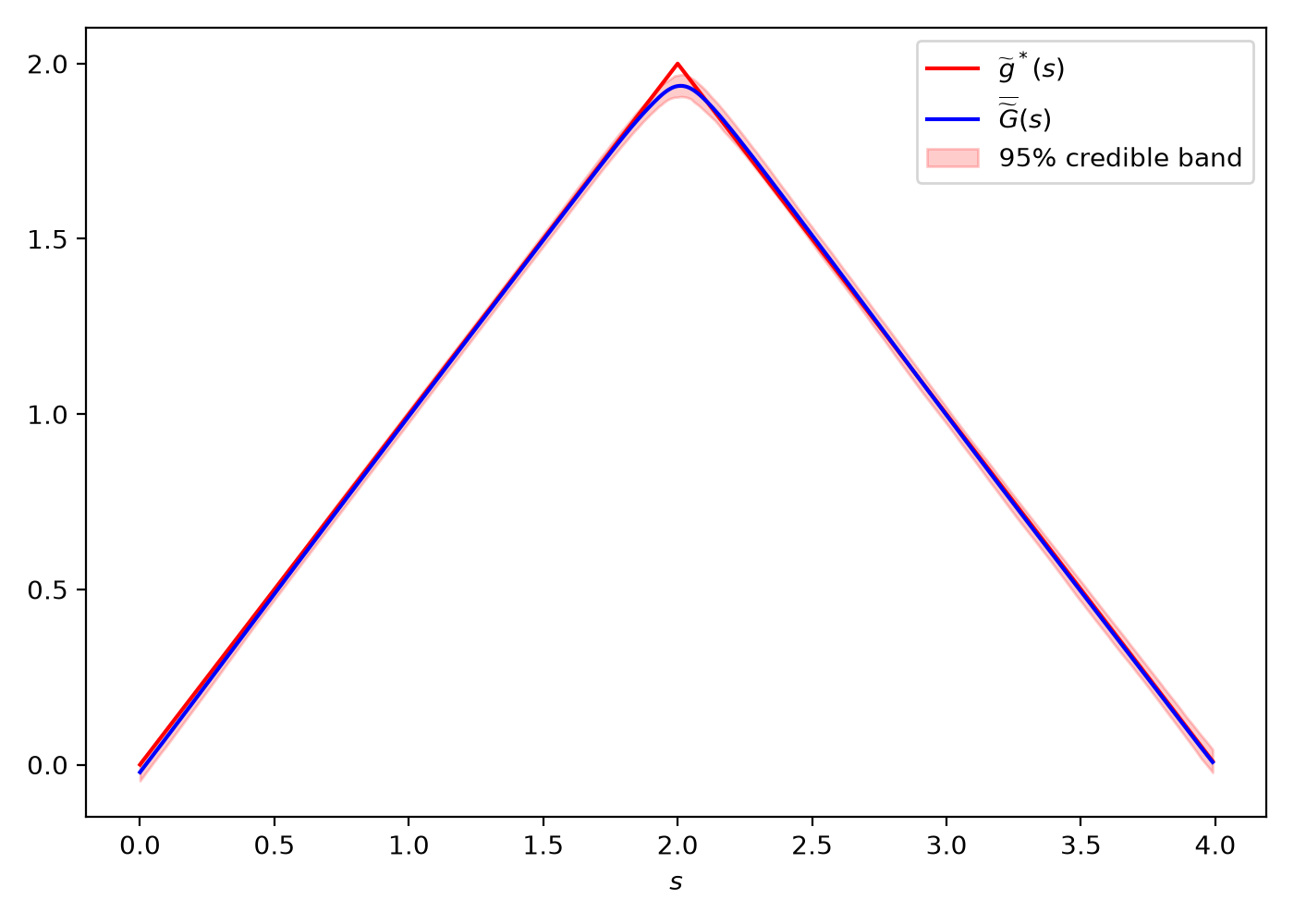}
    \caption{Boundary-data approximation with \(n_\Omega=800\) and \(n_\partial=200\). The figure shows the target data \(\widetilde g(s)\), the posterior mean \(\overline{\widetilde G}(s)\), and the pointwise \(95\%\) credible band.}
    \label{fig:single-run-boundary-trace}
\end{figure}

As shown in Figure~\ref{fig:single-run-source}, the posterior mean \(\bar F\) captures the main qualitative structure of the exact source \(f\). The top-left panel shows a smooth source profile with a central peak and decay toward the boundary, while the top-right panel shows that \(\bar F\) reproduces the same location, shape, and magnitude of this dominant feature. The bottom-left panel shows that the relative error remains small over most of \(\Omega\), with larger values appearing mainly near parts of the boundary. This is expected because the source values are smaller near the boundary, making the relative error more sensitive to small absolute discrepancies. The bottom-right panel shows that the posterior standard deviation \(\sigma_F\) is spatially nonuniform. The uncertainty is lower near the central region and higher near portions of the outer domain, indicating that the posterior samples are more tightly concentrated around the dominant interior structure of the source than near the boundary. All together, the source posterior provides an accurate mean reconstruction while also identifying regions where the approximation is less certain.

\begin{figure}[H]
    \centering
    \includegraphics[width=.7\textwidth]{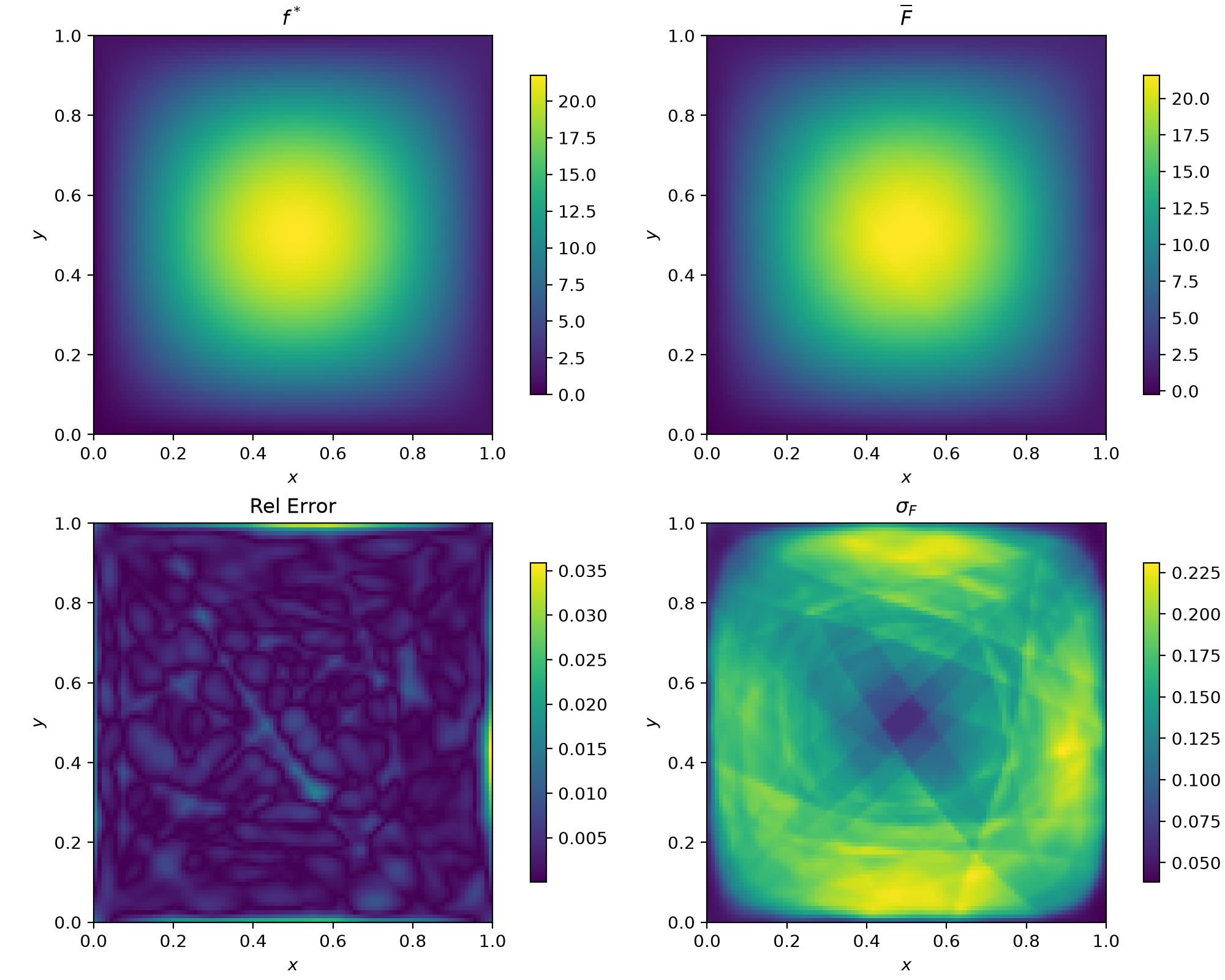}
    \caption{Interior source approximation with \(n_\Omega=800\) and \(n_\partial=200\). The panels show the exact source \(f\), the posterior mean \(\bar F\), the relative error, and the posterior standard deviation \(\sigma_F\).}
    \label{fig:single-run-source}
\end{figure}

\begin{figure}[H]
    \centering
    \includegraphics[width=0.5\textwidth]{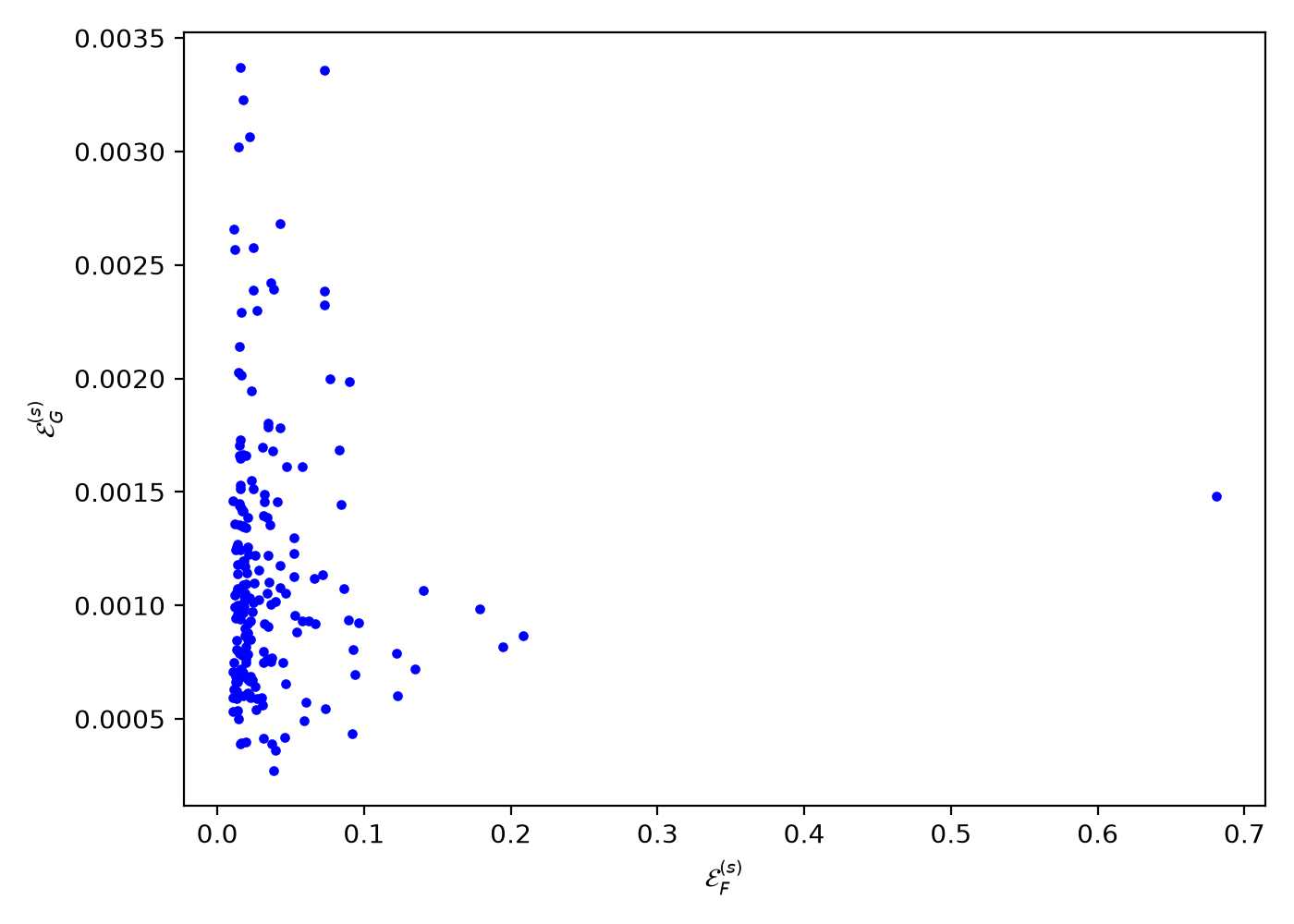}
   \caption{Samplewise error decomposition with \(n_\Omega=800\) and \(n_\partial=200\). The horizontal axis shows the source error \(\mathcal E_F^{(s)}\), and the vertical axis shows the boundary error \(\mathcal E_G^{(s)}\).}
    \label{fig:single-run-split-contributions}
\end{figure}

Figure~\ref{fig:single-run-split-contributions} shows the samplewise contributions \(\mathcal E_F^{(s)}\) and \(\mathcal E_G^{(s)}\) from the posterior draws. Most samples are concentrated near small values of both quantities, indicating that the posterior draws generally produce small source and boundary errors. The point cloud is more spread out along the \(\mathcal E_F^{(s)}\)-axis than along the \(\mathcal E_G^{(s)}\)-axis. This shows that the source contribution varies more across posterior samples, whereas the boundary contribution is more tightly concentrated. Thus, for this experiment, most of the variability in the operator-split physics-informed loss comes from the learned source term. The isolated point with large \(\mathcal E_F^{(s)}\) indicates that a small number of posterior draws can have noticeably larger source error, even when \(\mathcal E_G^{(s)}\) remains small. Taken together, the figure shows that the boundary component is stable across posterior samples, while the interior source component accounts for the dominant samplewise variation.

\begin{figure}[H]
    \centering
    \includegraphics[width=.7\textwidth]{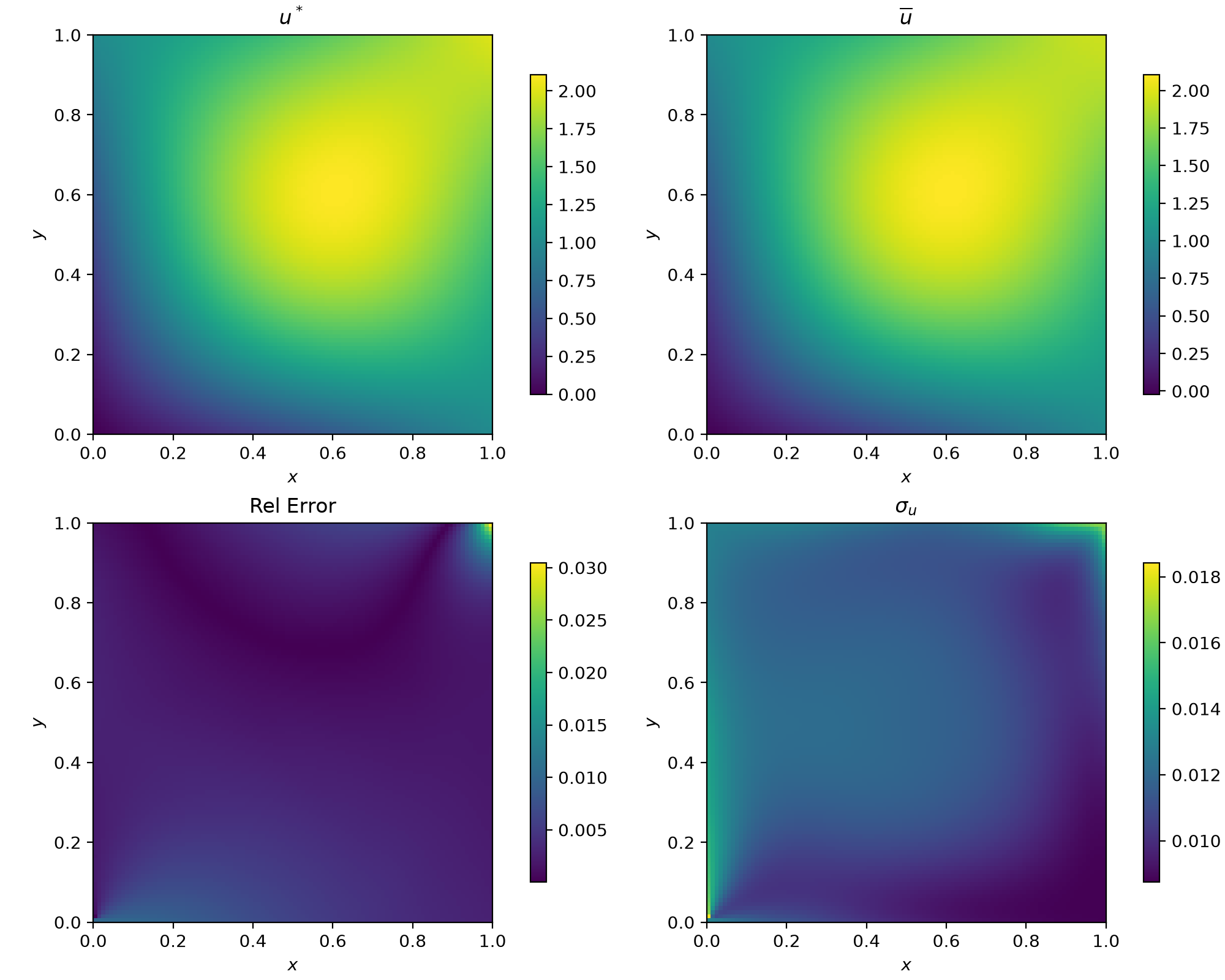}
    \caption{Posterior approximation of the solution with \(n_\Omega=800\) and \(n_\partial=200\). The panels show the exact solution \(u^\ast\), the posterior sample mean \(\bar u\), the relative error, and the posterior standard deviation \(\sigma_u\).}
    \label{fig:single-run-solution}
\end{figure}
The posterior solution approximation is presented in Figure~\ref{fig:single-run-solution}. The top-left panel gives the exact solution \(u^\ast\), which has a smooth profile with values increasing from the lower-left corner and attaining larger values toward the upper-right portion of the domain. The top-right panel shows that the posterior mean \(\bar u\) preserves this global structure, including the overall gradient, interior curvature, and boundary-driven variation. The bottom-left panel reports the relative error. The error remains small over most of \(\Omega\), with the largest values localized near the corners, especially close to the upper-right corner. This indicates that the posterior mean is accurate throughout the interior of the domain, with the most visible discrepancies confined to localized boundary and corner regions. The bottom-right panel reports the posterior standard deviation \(\sigma_u\). Its magnitude is small compared with the posterior uncertainty in the source term shown in Figure~\ref{fig:single-run-source}. This is consistent with the smoothing effect of the elliptic solution operator, since variability in the learned source and boundary data does not propagate directly into equally large variability in the solution. 

Finally, we would like to examine how the computational time changes with the number of retained posterior samples. In order to account for minor runtime fluctuations, each timing experiment was repeated three times under identical computational conditions. The mean runtime and corresponding standard deviation are reported in Table~\ref{tab:runtime-ablation}. 

\begin{table}[H]
    \centering
    \caption{Computational time for propagating different posterior ensemble sizes.}
    \label{tab:runtime-ablation}
    \vspace{.5em}
    \begin{tabular}{rrrrr}
        \toprule
        \(S\)
        & Mean time (s)
        & Time SD (s)
        & Time/sample (ms)
        & Samples/s \\
        \midrule
        200  & 4.4230  & 0.0041 & 22.115 & 45.218 \\
        400  & 8.8462  & 0.0063 & 22.115 & 45.217 \\
        800  & 17.6938 & 0.0143 & 22.117 & 45.214 \\
        1600 & 35.3807 & 0.0294 & 22.113 & 45.222 \\
        3200 & 70.7775 & 0.0571 & 22.118 & 45.212 \\
        \bottomrule
    \end{tabular}
\end{table}
We observe that the posterior propagation time increases essentially linearly with the ensemble size. Increasing the number of posterior samples from \(200\) to \(3200\), a factor of \(16\), increases the mean runtime from \(4.4230\) seconds to \(70.7775\) seconds, which is also approximately a factor of \(16\). A least-squares fit gives $T(S)=-0.0018+0.022118S,\, R^2=0.99999998$ where \(T(S)\) is measured in seconds. The average computational cost remains nearly constant at approximately
\(0.0221\) seconds, or \(22.1\) milliseconds, per posterior draw, corresponding to a throughput of about \(45.2\) posterior solutions per second. The runtime standard deviations are also small, with relative variations below \(0.1\%\) for every ensemble size. Throughout this paper, all our numerical experiments were performed on a single compute node equipped with one NVIDIA A40 GPU with \(48\) GB of memory, 8 CPU cores, and \(128\) GB of system memory.

\section{Conclusion and Discussion}
\label{sec:conclusion}

In this paper, we introduced an operator-split Bayesian framework for elliptic Dirichlet problems with noisy and imbalanced source and boundary observations. Unlike direct Bayesian PINN formulations that place Bayesian priors on neural-network representations of the solution \cite{YangMengKarniadakis2021,SunMukherjeeAtchade2024,ZhaoLu2026}, the proposed construction assigns independent BNN priors to the source term and the boundary data, and then propagates posterior samples through the elliptic solution operator. This decoupling allows the interior and boundary observations to be analyzed at their natural statistical dimensions, yielding separate \(d\)-dimensional and \((d-1)\)-dimensional contributions to the contraction radius. Up to logarithmic factors, the resulting upper bound matches the two-sample minimax lower bound of \cite{ZhaoLu2026}. The boundary sampling condition further identifies when the boundary contribution is no longer the limiting factor, so that the overall rate is governed primarily by the learned source term. In practice, the method avoids repeated automatic differentiation of a neural-network solution inside the physics-informed loss, which is a known bottleneck in residual-based PINNs, especially for higher-order differential operators \cite{SharmaShankar2022}. It also avoids smoothness restrictions on activation functions arising from the evaluation of higher-order derivatives of the neural-network solution \cite{WangLuSongHuang2023}. In our numerical implementation, the stiffness matrix is assembled once, and the reduced stiffness matrix obtained after imposing the Dirichlet boundary conditions is factorized once using a sparse LU factorization. Both the matrix and its factorization are reused across posterior samples; only the sample-dependent right-hand side and boundary contribution are updated from one draw to the next. This makes the posterior propagation step efficient while allowing the solution samples to inherit the stability and accuracy of the underlying discretization.

Our work has several limitations, which open several directions for future research. The analysis presented in the earlier sections is restricted to linear, uniformly elliptic Dirichlet problems with sufficiently regular coefficients, source terms and boundary data, together with Gaussian observation noise, and independent observation models. The cost of solving a deterministic PDE for each posterior draw, though modest in our experiments, may become significant in high dimensions, large-scale geometries, or larger posterior ensembles. Possible remedies include parallel domain-decomposition solvers \cite{SmithBjorstadGropp1996}, multilevel Monte Carlo strategies \cite{Giles2015}, and reduced-basis methods for PDEs \cite{QuarteroniManzoniNegri2016}. Theorem~\ref{thm:main-two-sample} gives a contraction result for the operator-split posterior, whereas \cite{ZhaoLu2026} considers a direct sparse prior on \(u\). Establishing an imbalanced two-sample contraction result for that direct-prior formulation remains open and would require estimates that control the interior residual and boundary penalty simultaneously under unequal sampling. Furthermore, while this study focused mainly on propagating uncertainty from the source and boundary data to the PDE solution, a natural extension is the corresponding Bayesian inverse problem, including the formulation of suitable observation models, conditions for identifiability and stability, and posterior contraction results for recovering unknown source terms, boundary data, or coefficients from noisy solution observations. Further extensions may include nonlinear and time-dependent PDEs, where the solution map and contraction analysis must account for nonlinear coupling and temporal evolution. Adaptive allocation of interior and boundary observations could be developed using the relative sizes of the two error contributions. Non-Gaussian likelihoods would broaden the framework to heavy-tailed noise, outliers, and model mismatch. Stronger boundary norms may yield contraction in stronger Sobolev spaces, while scalable posterior samplers and parallel deterministic solvers would improve efficiency for larger networks, finer discretizations, and larger posterior ensembles.

\bibliographystyle{plainnat}
\bibliography{refs}

\end{document}